\documentclass[11pt]{amsart}
\usepackage{geometry}               
\usepackage{graphicx}
\usepackage{comment}
\usepackage{amsmath}
\usepackage{amssymb}
\usepackage{pdfsync}
\usepackage{hyperref}
\usepackage{color}
\usepackage{epstopdf}
\usepackage{caption}
\def\R{\mathbb{R}}

\def\N{\mathbb{N}}
\def\Z{\mathbb{Z}}
\def\Co{\mathbb{C}}

\def\c{\mathfrak{c}}
\def\s{\mathfrak{s}}

\newtheorem{theorem}{Theorem}
\newtheorem{corollary}[theorem]{Corollary}

\newtheorem{proposition}[theorem]{Proposition}
\newtheorem{lemma}[theorem]{Lemma}

\title{Estimates for certain integrals of products of six Bessel functions.}
\author{Diogo Oliveira e Silva and Christoph Thiele}
\address{Hausdorff Center for Mathematics, Universit\"{a}t Bonn, 53115 Bonn, Germany.}
\email{dosilva@math.uni-bonn.de}
\email{thiele@math.uni-bonn.de}

\date{\today}                                           

\begin{document}
\begin{abstract}

We establish good numerical estimates for a certain class of integrals involving  
sixfold products of Bessel functions. We use relatively elementary methods.
The estimates will be used in the study of a sharp Fourier restriction inequality on the circle in \cite{CFOST}. 
\end{abstract}

\maketitle

\section{Introduction}

Let $(\mathbb{S}^1,\sigma)$ denote the unit circle in the plane equipped with its arc length measure. The companion paper \cite{CFOST} discusses partial progress towards understanding the optimal constant ${\bf C_{{\rm opt}}}$ in the endpoint Tomas-Stein adjoint restriction inequality \cite{T} on the circle:
\begin{equation}\label{Intro_TomasStein}
\|\widehat{f\sigma}\|_{L^6(\R^2)}\leq {\bf C_{{\rm opt}}} \,\|f\|_{L^2(\mathbb{S}^1)},
\end{equation}
where the Fourier transform of the measure $f\sigma$ is given by 
\begin{equation}\label{fsigmahat}
\widehat{f\sigma}(x)=\int_{\mathbb{S}^1} f(\omega)e^{-i x\cdot\omega} d\sigma_\omega,\;\;\;(x\in\R^2).
\end{equation}
It is conjectured that equality is attained in \eqref{Intro_TomasStein} 
when $f$ is a constant function.
For the constant function $f=\bf{1}$, the sixth power of the left-hand side of inequality \eqref{Intro_TomasStein} turns into the integral

\begin{equation}\label{alpha0}
(2\pi)^{7} \int_0^\infty J_0^6(r) rdr,
\end{equation}
where the Bessel function of order $n$, denoted $J_n$, is defined via the identity
\begin{equation}\label{definebessel}
\widehat{e^{in\cdot}\sigma}(x)= 2\pi (-i)^n J_n(|x|) e^{in\arg(x)}.
\end{equation}
Part of the analysis in \cite{CFOST} consists of a Fourier expansion 
of $f$ on the circle, and one needs estimates 
with rather precise numerical error bounds for  integrals  of the following two types:

\begin{equation}\label{firstintegral}
I_0=I_{0,m,n}:=\int_0^\infty J_{n+m}(r) J_n(r) J_m (r) J_0^3(r)rdr
\end{equation}
and
\begin{equation}\label{secondintegral}
I_1=I_{1,m,n}:=\int_0^\infty J_{n+m}(r) J_n(r) J_m (r) J_1^2(r)J_0(r)rdr.
\end{equation}
 The purpose of the present paper is to establish these estimates, summarized in the following theorem:

\begin{theorem}\label{thm}

Let $n\ge 2$ be some integer. Then each of the following
quantities is less than $0.002 n^{-4}$:

\begin{itemize}
\item[$(i)$] For $n\ge 7$:
$$
\left|I_{0,0,n}
-\frac{3}{4\pi^2}\frac{1}{n}+\frac{3}{32 \pi^2}\frac{1}{(n-1)n(n+1)}\right|,
$$
and for $n\ge 3$:
$$
\left|I_{1,0,n}-
\frac{1}{4\pi^2}\frac{1}{n}-\frac{3}{32\pi^2}\frac{1}{(n-1)n(n+1)}\right|.
$$
\item[$(ii)$] For any $n\ge 2:$
$$
\left|I_{0,2,n}-\frac{15}{64\pi^2}\frac{1}{n(n+1)(n+2)}\right|,
$$
$$
\left|I_{1,2,n}-\frac{9}{64\pi^2}\frac{1}{n(n+1)(n+2)}\right|.
$$
\end{itemize}

\noindent Moreover, each of the following
quantities is less than $0.0015 n^{-4}$:

\begin{itemize}
\item[$(iii)$] For $n\ge 4$:
$$
\left|I_{0,4,n}-\frac{1557}{1024\pi^2}\frac{1}{n(n+1)(n+2)(n+3)(n+4)}\right|,
$$
$$
\left|I_{1,4,n}-\frac{855}{1024\pi^2}\frac{1}{n(n+1)(n+2)(n+3)(n+4)}\right|.
$$
\item[$(iv)$] For even $m\geq 6$ and $n\ge m$:
$
\left|I_{0,m,n}\right|$ and $\left|I_{1,m,n}\right|$.
\end{itemize}
\end{theorem}

Thus Theorem \ref{thm} controls integrals of the two types
$I_0$ and $I_1$
for $n\ge 2$ and even $0\le m\le n$, with the exception of
the five cases $m=0$ and $n=2,3,4,5,6$ for $I_0$, and the two
cases $m=0$ and $n=2,3$ for $I_1$. It follows from Table
\ref{table:bigtable} below that, in these exceptional cases, the quantities
are still less than $0.01n^{-4}$, which provides information
about $I_0$ and $I_1$ with at least two percent relative accuracy.

It follows from the methods of this paper, or alternatively
from general principles, that such a result holds with bounds
$cn^{-4}$ for any positive number $c$ in place of $0.002$ or $0.0015$, 
and with some finite set of exceptions. The point of Theorem \ref{thm} is
to  narrow down these exceptions precisely for the specific numbers $c=0.002$ (for $m=0,2$) and $c=0.0015$ (for $m\geq 4$).
Slightly better numerical estimates  are listed in Sections 
\ref{sec:puttingtogether} and \ref{sec:numerical} for the various cases,
but for simplicity we do not reproduce all of them here.

Our methods apply to obtain a more general set of estimates than the ones listed
in Theorem \ref{thm}, but  we focus  on the stated estimates which are
needed in \cite{CFOST}.
There exists a very satisfactory theory of similar integrals of products
of two Bessel functions, see for example Lemmata \ref{kapteynlemma} and \ref{WeberSchafheitlin} below, and a still explicit but 
substantially more complicated theory for integrals of products of four Bessel 
functions. 
While  integrals of sixfold products of Bessel functions still fall into the class of functions for which explicit symbols have been introduced in the theory 
of hypergeometric functions and their generalizations, we  do not know how to 
obtain our rather accurate numerical bounds in a much easier way than by the elementary
but somewhat laborious approach presented in this paper.

Our approach is to expand four of the six Bessel function factors, namely
those four with the lowest orders, into their asymptotic expansions. This will reduce the integrals in question to core integrals
of the type

\begin{equation}\label{mainintegrals}
\int_0^\infty J_n(r) J_{n+m}(r) \sin(\ell r)r^{-k}dr,\ \int_0^\infty J_n(r) J_{n+m}(r) \cos(\ell r)r^{-k}dr
\end{equation}
for $\ell=0,\pm2,\pm 4$. For these integrals, one has good information as in
Lemmata \ref{2rlemma}, \ref{WeberSchafheitlin}, \ref{L2}.
In more detail, the paper is organized as follows. In Section \ref{sec:Background}, we review the theory of Bessel functions inasmuch as it is useful for our purposes. In particular, we establish the aforementioned lemmata, together with asymptotic expansions with precise control on the error terms. 
In Section \ref{sec:Useful}, we prove some useful estimates for binomial coefficients, the Gamma function, and the coefficients that arise in the various asymptotic expansions. The analytic part of the proof of Theorem \ref{thm} begins in Section \ref{sec:Part1}, where we asymptotically expand the functions $J_0$ and $J_1$. Section \ref{sec:Part2} accomplishes the same for the function of next lowest order, namely $J_m$. Finally, Section \ref{sec:Part3} is devoted to the analysis of the core integrals. The estimates from Sections \ref{sec:Part1}$-$\ref{sec:Part3} are then assembled together in Section \ref{sec:puttingtogether}. The approach works for $n\ge 20$, and so for $n<20$ we numerically estimate
the integrals; this is the content of the final Section \ref{sec:numerical}.

We close this discussion with a brief illustration of the difficulty involved.
Figure \ref{fig:sixbessel} depicts the plot of the integrand of $I_{1,6,9}$
between $r=0$ and $r=100$.
One observes an initial region until about $r=n=9$ where the function is very small. Then one sees a region with fairly erratic behaviour until about $r=n^2=81$. Past $r=81$,
one sees a more repetitive behaviour where one has good asymptotic control. 
The asymptotic region yields a positive contribution to the desired integral, 
which is in general of the order $n^{-2}$.
The erratic region yields a negative contribution which nearly cancels the
positive part from the asymptotic region. In question is a very good numerical
control of the order $n^{-4}$ of the small difference. The main tools to capture this cancellation are the algebraic identities from Lemma \ref{2rlemma} and an exact orthogonality formula due to Kapteyn \cite{K} which can be found in  Lemma \ref{kapteynlemma}.

\begin{figure}[htb]
  \centering
  \includegraphics[height=9cm]{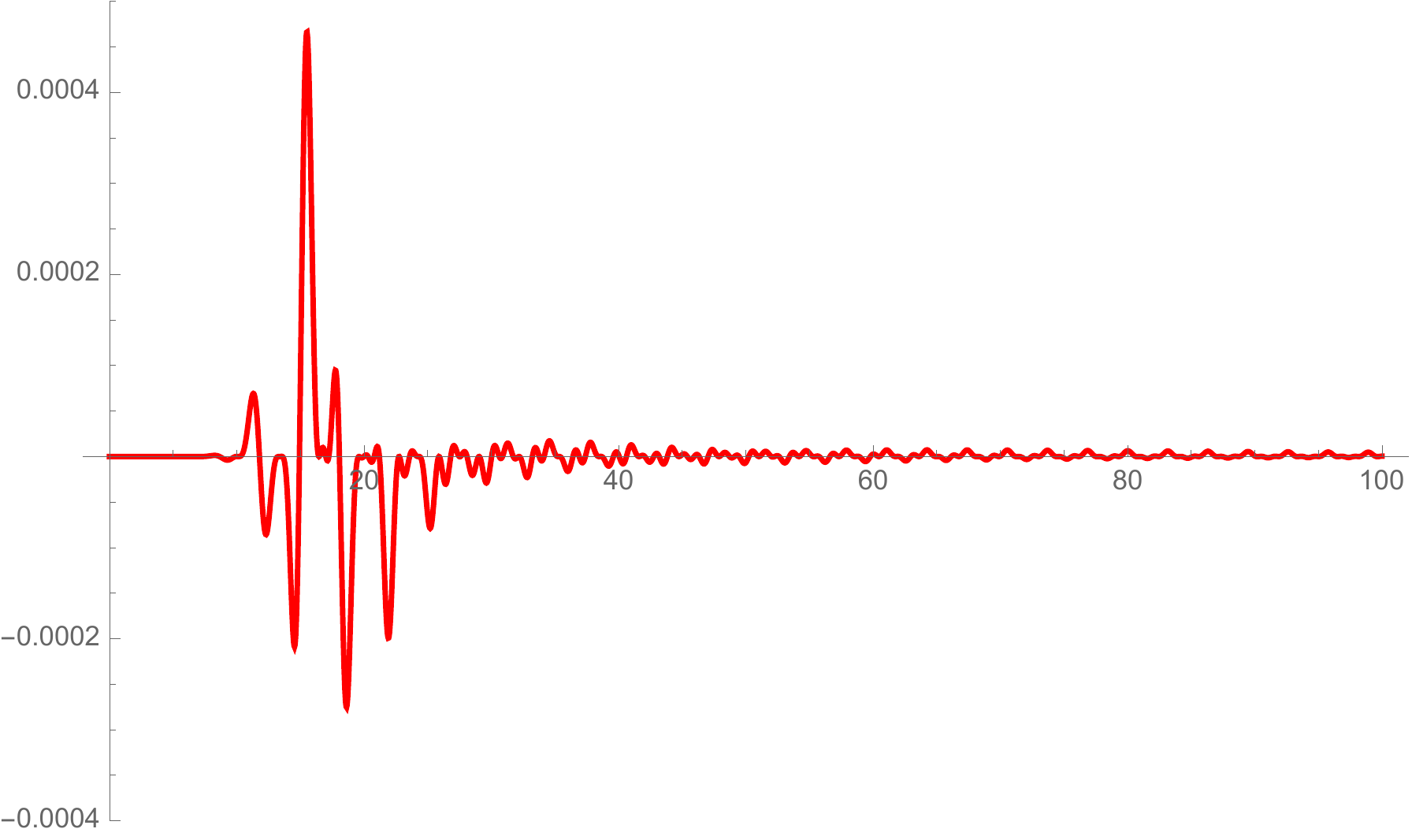}
  \caption{Plot of the function $J_{15}(r)J_{9}(r)J_{6}(r)J_{1}^2(r)J_0(r)r$ for $0\leq r\leq 100$.}
\label{fig:sixbessel}
\end{figure}

\vspace{.5cm}

\noindent {\bf Acknowledgements.} 
The software {\it Mathematica} was heavily used in the brainstorming phase of the research project, as well as in the numerical part of the paper. We are thankful to Emanuel Carneiro for helpful discussions during the preparation of this work, and to Pavel Zorin-Kranich for pointing out an improvement to the first version of our {\it \mbox{Mathematica}} code. Finally, we would like to thank both the  
Hausdorff Center for Mathematics and the Hausdorff Institute for Mathematics for support.


\section{Background on Bessel functions}\label{sec:Background}

We rewrite the definition \eqref{definebessel} of the Bessel function in the 
form of a Bessel integral which  is the starting point in \cite[p.~338]{S}. For
$n\in \Z$ and $z\geq 0$, we claim that

\begin{equation}\label{steinbesselintegral}J_n(z)=\frac{1}{2\pi}\int_{-\pi}^{\pi} e^{i z \sin\theta} e^{-i n \theta} d\theta.\end{equation}
More precisely, replacing $\theta=\omega+\pi/2$ in \eqref{steinbesselintegral}
and using even symmetry of the cosine we obtain for the right-hand side of \eqref{steinbesselintegral}:

\begin{equation}\label{cosinebesselintegral}
\frac{(-i)^n}{2\pi}\int_{-\pi}^{\pi} e^{i z \cos\omega} e^{-i n \omega} d\omega=\frac{(-i)^n}{\pi}\int_{0}^{\pi} e^{i z \cos\omega} \cos(n \omega) d\omega,
\end{equation}
from which the equivalence of \eqref{steinbesselintegral} and \eqref{definebessel} is evident.

The Bessel function, defined via \eqref{steinbesselintegral} for general $z\in\Co$, is an entire function. From \eqref{steinbesselintegral} we obtain the estimate  
\begin{equation}\label{EstJnIm}
|J_n(z)|\leq e^{|\Im (z)|}.
\end{equation}
Differentiation under the integral sign in \eqref{steinbesselintegral}
and  integration by parts yield the following recurrence relations
in the sense of meromorphic functions:
\begin{align}
J_{n-1}(z)-J_{n+1}(z)&=2J_n'(z),\label{Recurrence0}\\
J_{n-1}(z)+J_{n+1}(z)&=\frac{2n}{z} J_n(z).\label{Recurrence1}
\end{align}

A different representation of the Bessel function is the Poisson integral,
which contains a power of a trigonometric function rather than a power of an 
exponential function:
\begin{equation}\label{poissonintegral}{J}_n(z)=\frac{(z/2)^n}{\Gamma(n+1/2)\Gamma(1/2)} 
\int_0^\pi \cos(z \cos(\theta))\sin^{2n}(\theta)\, d\theta\ .\end{equation}
Here we use the Gamma function 
$$\Gamma(s):=\int_0^\infty e^{-t}t^{s-1}dt\ ,$$
which satisfies the functional equation $s\Gamma(s)=\Gamma(s+1)$
and thus meromorphically extends the factorial, that is $\Gamma(n+1)=n!$ for
natural numbers $n$. We mainly need the Gamma function for half-integer 
values, which can be expressed as
\begin{equation}\label{HalfIntegerGamma}
\Gamma\Big(\frac{1}{2}+n\Big)=\frac{(2n)!}{4^n n!}\sqrt{\pi}
\textrm{ and }
\Gamma\Big(\frac{1}{2}-n\Big)=\frac{(-4)^n n!}{(2n)!}\sqrt{\pi}.
\end{equation}
This can be recursively verified from $\Gamma(\frac 12)=\sqrt{\pi}$, 
which in turn can be read from the well-known property

\begin{equation}\label{HalfSine}
\sin(\pi x)\Gamma(x)\Gamma(1-x)=\pi. 
\end{equation}
The latter can be seen by verifying periodicity of the left-hand side 
together with growth estimates which force the left-hand side to be constant.

To see equivalence of the Poisson integral representation \eqref{poissonintegral} with \eqref{steinbesselintegral}, one verifies the case
$n=0$ by substitution and then verifies by partial integration 
that the Poisson integral also satisfies the recursion relations
\eqref{Recurrence0} and \eqref{Recurrence1}.  Combining these two second order recurrence relations
into a first order relation between $J_n$ and $J_{n+1}$, equivalence of
the two integral representations follows recursively by 
a uniqueness result for ordinary differential equations, where we use that both
integral representations vanish for $z=0$ and $n>0$.

From the Poisson integral representation one sees the following estimate from \cite[\S 3.31(1), p.~49]{W}, useful for small $z$:
\begin{equation}\label{wholeplaneJn}
|J_n(z)|\leq \frac{|z|^n e^{|\Im{z}|}}{2^n n!}\ ,
\end{equation}
where we have used
$$\frac 1{n!}=\frac 1{\Gamma(n+1/2)\Gamma(1/2)} 
\int_0^\pi \sin^{2n}(\theta)\, d\theta\ ,$$
which one proves by induction on $n$ using integration by parts.

We turn to the core integrals \eqref{mainintegrals}.
The case $\ell=\pm 2$ will be the most pleasant to deal with via the following lemma:

\begin{lemma}\label{2rlemma}
Let $0\le n,m$ have the same parity, and $1\le k\le n+m$. Then if $k$ is even 
$$\int_0^\infty J_n(r)J_m(r) r^{-k} \cos(2r)\, dr=0\ ,$$
and if $k$ is odd
$$\int_0^\infty J_n(r)J_m(r) r^{-k} \sin(2r)\, dr=0\ .$$
\end{lemma}
\begin{proof}
By the parity assumption, we may extend the integrals to the full real line. It then suffices to show that the Fourier transform
$$\int_{-\infty}^\infty J_n(r)J_m(r) r^{-k} e^{-i \xi  r} \, dr$$
vanishes at $\xi=2$. 
Substituting $\xi=\cos\omega$ on the right-hand side of \eqref{cosinebesselintegral} yields
$$J_n(z)=\frac{(-i)^n}{\pi}\int_{-1}^{1} e^{i z \xi} \frac{\cos(n \arccos \xi)}{\sqrt{1-\xi^2}} d\xi\ .$$ 
Hence we see that $J_n$ is the Fourier transform of the function
$$B_n(\xi)=  \frac{(-i)^n}{\pi} T_n(\xi)(1-\xi^2)^{-1/2}1_{[-1,1]}(\xi)\ ,$$ where $T_n$ denotes the Chebysheff polynomial
$T_n(\xi):=\cos(n\arccos \xi)$.

We first consider the case $k\le n$ in the lemma. Then the Fourier
transform $B_n^{(-k)}$ of $J_n(r)r^{-k}$ is still supported on $[-1,1]$ since 
$B_n$ has vanishing moments of orders $0,1,\ldots,k-1$. This can be deduced from \eqref{poissonintegral}. Seen as the convolution of an $L^{p'}$ function with an $L^p$ function for $p=2^-$, the function
$B_n^{(-k)}*B_m$  is  continuous.
 Since it is also
supported on the interval $[-2,2]$,  it must vanish at $\xi=2$. This proves the lemma in case $k\le  n$.
If $n< k\le  n+m$, we distribute some powers of $r$ over $J_m$ and argue similarly.
\end{proof}

The understanding of the dominant case $\ell=0$ of the core integrals 
\eqref{mainintegrals} begins with Kapteyn's identity, proved in a delightful two page paper
\cite{K}.
\begin{lemma}[\cite{K}]\label{kapteynlemma}
If $n,m\geq 0$ and $n+m\neq 0$, then
\begin{equation}\label{kapteynidentity}
\int_0^\infty J_n(r)J_m(r) r^{-1} dr
=\frac 2\pi \frac{\sin \frac{m-n}2 \pi}{m^2-n^2}
\end{equation}
with the following natural interpretation in case $n=m$:
$$\int_0^\infty J_n(r)^2 r^{-1} dr=\frac 1 {2n}\ .$$
\end{lemma}
Note in particular that \eqref{kapteynidentity} vanishes if $m-n$ is a nonzero even integer. Moreover, identities \eqref{HalfIntegerGamma}, \eqref{HalfSine}, and some algebra yield

$$\frac 2\pi \frac{\sin \frac{m-n}2 \pi}{m^2-n^2}
=\frac{2^{-1}\Gamma(\frac{m+n}2)}{\Gamma(\frac{m+n+2}2)\Gamma(\frac{n-m+2}2)\Gamma(\frac{m-n+2}2)}.$$
From Kapteyn's identity, one obtains the following more general result.
\begin{lemma}\label{WeberSchafheitlin} 
If $0\le n, m$ and $1\le k\le n+m$, then
\begin{equation}\label{wsintegral}\int_0^\infty J_n(r)J_{m}(r) r^{-k} dr=
\frac{2^{-k}\Gamma(k) \Gamma(\frac{m+n+1-k}2)}{\Gamma(\frac{m+n+1+k}2)
\Gamma(\frac{n-m+k+1}2)\Gamma(\frac{m-n+k+1}2)}\ .
\end{equation}
\end{lemma}
\begin{proof}
The identity is true for $k=1$ by Lemma \ref{kapteynlemma}.
Let $k\ge 1$ be given and assume that identity \eqref{wsintegral} is true
for this particular $k$. To prove the identity for $k+1$ we denote the integral
in \eqref{wsintegral} by $\mathcal{I}_{n,m,k}$.
Then, by the second recursion in \eqref{Recurrence1} and using the induction hypothesis, we have that
$$\mathcal{I}_{n,m,k+1}=\frac 1{2n} (\mathcal{I}_{n-1,m,k}+\mathcal{I}_{n+1,m,k})$$
$$=\frac{2^{-(k+1)}n^{-1}\Gamma(k) \Gamma(\frac{m+n-k}2)}{ \Gamma(\frac{m+n+k}2)
\Gamma(\frac{n-m+k}2)\Gamma(\frac{m-n+k+2}2)}+
\frac{2^{-(k+1)}n^{-1}\Gamma(k) \Gamma(\frac{m+n+2-k}2)}{\Gamma(\frac{m+n+2+k}2)
\Gamma(\frac{n-m+k+2}2)\Gamma(\frac{m-n+k}2)}=$$
$$\frac{2^{-(k+1)}\Gamma(k+1) \Gamma(\frac{m+n-k}2)}{ \Gamma(\frac{m+n+2+k}2)\Gamma(\frac{n-m+k+2}2)\Gamma(\frac{m-n+k+2}2)}
\frac{(m+n+k)(n-m+k)+(m+n-k)(m-n+k)}{4 n k}$$
The second fraction in the last line is equal to $1$, and this proves the inductive step.
\end{proof}
Note in particular that if $k$ is odd and satisfies $k<|n-m|$, then \eqref{wsintegral} vanishes since the function $1/\Gamma$ has zeros at $s=0,-1,-2,\ldots$
Note also that three of the Gamma factors, the one with argument $k$ in
the numerator and the two involving the difference $m-n$ in the denominator,
typically form a  binomial coefficient, which can be roughly
estimated by $2^k$.
An alternative approach to Lemma \ref{WeberSchafheitlin} is the integration theory of Weber and Schafheitlin as outlined in \cite[\S 13.24, p.~398]{W}.

The case $\ell=\pm 4$ of the core integrals \eqref{mainintegrals} 
gives small error terms which we estimate with the following lemma.
\begin{lemma}\label{L2}
Let $0\le n,m$ 
and $1\le k<n+m$. Then:
$$\Big|\int_0^\infty J_n(r)J_{m}(r) r^{-k} e^{4ir} {dr}\Big|
\leq
\frac{2^{k-1}}{4^{n+m}} \frac{(n+m-k)!}{n!m!}.$$
\end{lemma}

\begin{proof}

We estimate this integral by the descent method, 
changing the contour of integration to the contour consisting
of a line segment from $0$ to $iN$ for some large $N$, then a line segment 
from $iN$ to $N$, and then a ray from $N$ to $\infty$ along the real axis.
Only the first integral provides a substantial contribution, since the next two 
segments give contributions that tend to $0$ as $N$ tends to infinity.

Along the first segment of the contour the integral can be estimated using \eqref{wholeplaneJn} by

\begin{align*}
\int_0^\infty |J_n(ix)||J_{m}(ix)| x^{-k} e^{-4x} {dx}&
\leq \frac{1}{n!m!}\frac{1}{2^{n+m}} \int_0^\infty  x^{n+m-k}e^{-2x} {dx}\\
&=\frac{2^{k-1}}{4^{n+m}} \frac{(n+m-k)!}{n!m!}.
\end{align*}
The integral over the second part of the contour is estimated using \eqref{EstJnIm} by

$$\sqrt{2} \int_0^N |J_n(x+i(N-x))||J_{m}(x+i(N-x))| |x+i(N-x)|^{-k} e^{-4(N-x)} {dx}$$ $$ \leq 
\sqrt{2} \int_0^N (N/\sqrt{2})^{-k} e^{-2(N-x)} {dx} \le 2^{k+1}N^{-k}\ , $$
which tends to 0 as $N\rightarrow\infty$. 
The integral over the third piece of the contour is estimated using the 
fact that the functions $J_n$ and $J_m$ are in $L^4$ as Fourier transforms 
of $L^{4/3}$ functions and thus the continuous function
$J_n(r)J_m(r)r^{-1}$ is in $L^1$. It then follows from the dominated convergence theorem that

$$\lim_{N\to \infty}\int_N^\infty |J_n(r) J_{n+m}(r)| r^{-1} {dr}=0\ .$$
Adding the contour integrals and letting $N\to \infty$ 
yields the desired bound.
\end{proof}

To arrive at the core integrals, we need asymptotic expansions of the 
Bessel functions near infinity as in the following lemma.
\begin{lemma}\label{AsymptoticBessel}
Let $n\in\N$ 
and $\Re(z)>0$. Let $\omega_n=z-{n\pi}/{2}-{\pi}/{4}$. Let $\ell\in\N$ be such that $\ell\geq \max\{n-1/2,1\}$. If $\ell$ is even, then

\begin{equation}\label{evenPQsum}
J_n(z) =\Big(\frac{2}{\pi z }\Big)^{\frac1 2}\Big[(\cos\omega_n) \left(\sum_{k=0}^{\frac{\ell}2-1} (-1)^k\frac{a_{2k}(n)}{z^{2k}}\right)-(\sin\omega_n)\left(\sum_{k=0}^{\frac{\ell}2-1} (-1)^k\frac{a_{2k+1}(n)}{z^{2k+1}}\right)\Big]+R_{n,\ell}(z).
\end{equation}
If $\ell$ is odd, then

\begin{equation}\label{oddPQsum}
J_n(z) =\Big(\frac{2}{\pi z }\Big)^{\frac1 2}\Big[(\cos\omega_n) \left(\sum_{k=0}^{\frac{\ell-1}{2}} (-1)^k\frac{a_{2k}(n)}{z^{2k}}\right)-(\sin\omega_n)\left(\sum_{k=0}^{\frac{\ell-3}{2}} (-1)^k\frac{a_{2k+1}(n)}{z^{2k+1}}\right)\Big]+{R}_{n,\ell}(z).
\end{equation}
Here the coefficients $a_j(n)$ are defined by

\begin{equation}\label{coeffajn}
a_j(n)=\frac{\Gamma(n+j+\frac 1 2)}{\Gamma(n-j+\frac 12) j! 2^j},
\end{equation}
and  the remainder function ${R}_{n,\ell}$ satisfies the bounds 
$$|R_{n,\ell}(z)|\leq 
\Big(\frac 2{\pi |z|}\Big)^{\frac1 2}
\frac{\Gamma(n+\ell+\frac1 2)}{|\Gamma(n-\ell+\frac1 2)|\ell !}
\left( \frac{|z|}{\Re(z)}\right)^{\ell-n+\frac1 2}\cosh(\Im(z))\left(\frac{1}{2|z|}\right)^\ell.$$
\end{lemma}

\begin{proof}[Proof]
We expand on the proof outlined in Watson \cite[\S 7.3, p.~205]{W}.
The change of variables $t=\cos\theta$ turns the Poisson integral \eqref{poissonintegral} into
$${J}_{n}(z)=\frac{(z/2)^{n}}{\Gamma(n+1/2)\Gamma(1/2)} 
\int_{-1}^{1} \cos(z t)(1-t^2)^{n-1/2} \, dt\ .$$
We split 
$$J_n(z)=\frac 12 (J_n^+(z)+J_n^-(z)),$$
with
$${J}_n^+(z)=\frac{(z/2)^n}{\Gamma(n+1/2)\Gamma(1/2)} 
\int_{-1}^{1} e^{iz t}(1-t^2)^{n-1/2} \, dt$$
and $J_n^-(\overline{z})=\overline{J_n^+({z})}$.
Now we change the contour in the integral for $J_{n}^+$ 
towards a $\Pi$-shaped contour consisting
of the line segment from $-1$ to $-1+iN$, followed by the line segment 
from $-1+iN$ to $1+iN$, and then the line segment from $1+iN$ to $1$. 
On that contour as well as on its convex hull, 
we have $\Re(1-t^2)\ge 0$ with equality only at the end points of the
contour. Hence we may choose the continuous branch of the square root 
function on the slit plane $\Co\setminus (-\infty,0)$ which takes positive values on the positive real axis.  
For simplicity of notation, let us introduce the half-integer $\nu:=n-1/2$. The integral over the first segment of the contour becomes

\begin{align*}
\frac{(z/2)^{\nu+1/2}}{\Gamma(\nu+1)\Gamma(1/2)} 
\int_0^{N} e^{-iz}e^{-zs}(2is+s^2)^{\nu} i\, ds
&=\frac{e^{-iz}}
{(2\pi z)^{1/2}\Gamma(\nu+1)} 
\int_0^{zN} e^{-u}u^{\nu}(i+\frac{u}{2z})^{\nu} i\, du\\
&=\frac{e^{i(-z+(\nu+1)\pi/2)}}{(2\pi z)^{1/2}\Gamma(\nu+1)} 
\int_0^{zN} e^{-u}u^{\nu}(1-\frac{iu}{2z})^{\nu}\, du\ ,
\end{align*}
where in the first identity we replaced $zs$ by $u$
and rotated the contour towards the line segments from $0$ to $zN$,
and used $\Gamma(1/2)=\pi^{1/2}$. In the second identity, we 
first pulled an integer power of $i$ out of the integral and
and then pulled half a power of $i$ out of the integral without leaving 
the domain of definition of the chosen branch of the square root function.
If $\Re(z)>0$, then this last integral has a limit as $N\to \infty$.
Similarly, the integral over the third line segment becomes  

$$- \frac{(z/2)^{\nu+1/2}}{\Gamma(\nu+1)\Gamma(1/2)} 
\int_0^{N} e^{iz}e^{-zs}(-2is+s^2)^{\nu} i\, ds
=\frac{e^{i(z-(\nu+1)\pi/2}
}{(2\pi z)^{1/2}\Gamma(\nu+1)} 
\int_0^{zN} e^{-u}u^{\nu}(1+\frac{iu}{2z})^{\nu} \, du\ .$$
If $\Re(z)>0$, then the limit again exists. Moreover,  the
integral over the second line segment tends to $0$ as $N\to \infty$. We therefore obtain
$$J_n^+(z)=\frac{(2\pi z)^{-1/2}}{\Gamma(\nu+1)}
\int_0^\infty e^{-u}u^{\nu}\left[
e^{-i\omega_n}
(1-\frac{iu}{2z})^{\nu}
+e^{i\omega_n}(1+\frac{iu}{2z})^{\nu}\right]\, du\ ,$$
where $\omega_n=z-{n\pi}/{2}-{\pi}/{4}=z-(\nu+1)\pi/2$, and the contour
has been changed to one along the real axis.

The function
$(1+\alpha u)^\nu$ is infinitely differentiable in $u\in [0,\infty)$
for any $\alpha\in \Co\setminus (-\infty,0]$. Taylor's expansion 
gives, for any $\ell\ge 1$,
 
$$ (1+\alpha u )^{\nu} = 
\sum_{k=0}^{\ell-1}\frac{\Gamma(\nu+1)}{\Gamma(\nu+1-k)}
\frac{\left(\alpha u\right)^{k}}{k!}
+ r   \frac{u ^{\ell}}{\ell !},$$
where $r$ is the $\ell$-th derivative of the function $u\mapsto (1+\alpha u)^\nu$ at some point $u_0\in [0,u]$.
If $\ell>\nu$, then this derivative is proportional to a negative power of the function
and thus attains its maximum at the point where the value of the function comes 
closest to the origin. This point equals $u=-\Re(\alpha)/|\alpha|^2$ and the value of
the $\ell$-th derivative there equals 
$$\frac{\Gamma(\nu+1)}{\Gamma(\nu+1-\ell)}
\alpha ^{\ell} \Big(1-\frac {\Re(\alpha)}{\overline{\alpha}}\Big)^{\nu-\ell}\ .$$
Hence we can write for the remainder term
$$r \frac{u ^{\ell}}{\ell !}=
\theta(u) \frac{\Gamma(\nu+1)}{\Gamma(\nu+1-\ell)}
\left|1-\frac { \Re(\alpha)}{\overline{\alpha}}\right|^{\nu-\ell}
\frac{\left(\alpha u \right)^{\ell}}{\ell !}$$
for some $|\theta(u)|\le 1$ that also depends on $\alpha$.
Note that
$$|1-\frac {\Re(\alpha)}{\overline{\alpha}}|=\frac{|\overline{\alpha}-\Re(\alpha)|}{|\alpha|}=
\frac{|\Im(\alpha)|}{|\alpha|},$$
and the latter equals $\Re(z)/|z|$ if $\alpha=\pm i/(2z)=\pm i\overline{z}/(2|z|^2)$.
For $\Re(z)>0$, we have that $1\pm\frac{ iu}{2z}$ is not real, hence we may
insert the Taylor expansion into the integral for $J_n^+$, and obtain for $J_n^+(z)$ the expression

\begin{align*}
&(2\pi z)^{-1/2}
\sum_{k=0}^{\ell -1} \frac{1}{\Gamma(\nu+1-k)k!} 
\int_0^\infty e^{-u}u^{\nu+k}\left[
e^{-i\omega_n} \left(\frac{-i}{2z}\right)^k 
+e^{i\omega_n}\left(\frac{i}{2z}\right)^k \right]\, du+R(z)\\
&=(2\pi z)^{-1/2}
\sum_{k=0}^{\ell -1} \frac{\Gamma(\nu+k+1)}{\Gamma(\nu+1-k)k!} 
\left[
e^{-i\omega_n} \left(\frac{-i}{2z}\right)^k 
+e^{i\omega_n}\left(\frac{i}{2z}\right)^k \right]+R(z)\ ,
\end{align*}
where the remainder term $R=R_{n,\ell}$ satisfies
$$|R_{n,\ell}(z)|\le  
\Big(\frac 2{\pi |z|}\Big)^{1/2}
\frac{\Gamma(\nu+\ell+1)}{|\Gamma(\nu+1-\ell)|\ell !}
\left( \frac{|z|}{\Re(z)}\right)^{\ell-\nu}\cosh(\Im(z))\left(\frac{1}{2|z|}\right)^\ell \ .$$
We split the summation into even integers and odd integers to obtain

\begin{align*}
J_n^+(z)
&=
\Big(\frac{2}{\pi z}\Big)^{1/2}
\sum_{\substack{0\le k<\ell: \\ k\ \rm even}} (-1)^{k/2}\frac{\Gamma(\nu+k+1)}{\Gamma(\nu+1-k)k!2^k} 
\cos(\omega_n) \left(\frac{1}{z}\right)^k\\
&-
\Big(\frac{2}{\pi z}\Big)^{1/2}
\sum_{\substack{0\le k<\ell: \\ k\ \rm odd}} (-1)^{(k-1)/2}\frac{\Gamma(\nu+k+1)}{\Gamma(\nu+1-k)k!2^k} 
\sin(\omega_n) \left(\frac{1}{z}\right)^k + R_{n,\ell}(z).
\end{align*}
This completes the proof of the lemma. 
 \end{proof}

We specify our findings into  explicit first order asymptotics 
valid for sufficiently large $z$ near the positive real axis.
We have with the notation from the above proof:

\begin{corollary}\label{zlargernsquare}
Let $z\in\Co$ be such that $\Im(z)<\Re(z)$. Then 
$$\left|J_{0}^\pm(z)- \left(\frac{2}{\pi z}\right)^{1/2}\cos(\omega_0)\right|\le 
\frac 1{8|z|}  \left(\frac{2}{\pi |z|}\right)^{1/2}\cosh(\Im (z))
\left(\frac{|z|}{|\Re(z)|}\right)^{3/2}\ ,
 $$
$$\left|J_{1}^\pm(z)- \left(\frac{2}{\pi z}\right)^{1/2}\cos(\omega_1)\right|\le 
\frac 3{8|z|}  \left(\frac{2}{\pi |z|}\right)^{1/2}\cosh(\Im (z))
\left(\frac{|z|}{|\Re(z)|}\right)^{1/2}\ ,
 $$
and if $n>1$ and $\Re(z)>n^2$, then
$$\left|J_{n}^\pm(z)- \left(\frac{2}{\pi z}\right)^{1/2}\cos(\omega_n)
\right|\le 
\left(\frac{2}{\pi |z|}\right)^{1/2} 
\frac{n^2}{|z|}\cosh(\Im (z)) 
\left(\frac{|z|}{|\Re(z)|}\right)^{1/2}\ .$$

\end{corollary}
\begin{proof}
Following the previous argument, only the last inequality requires explanation.
We choose $\ell=n=\nu+1/2$, and apply the above expansion with the observation that
$$\frac{\Gamma(\nu+1+k)}{\Gamma(\nu+1-k)}\le \Big(\nu+\frac 1 2\Big)^{2k}$$
for $k\le \nu$. It follows that
$$\left|J_{n}^\pm(z)- \left(\frac{2}{\pi z}\right)^{1/2}\cos(\omega_n)\right|\le
\left(\frac{2}{\pi |z|}\right)^{1/2} \sum_{k=1}^{\ell-1} n^{2k} \cosh(\Im(z))(2|z|)^{-k}+R
$$
$$\le 
\left(\frac{2}{\pi |z|}\right)^{1/2} \cosh(\Im(z))\frac{n^2}{|z|}(1-2^{1-\ell}) +R
\le 
\left(\frac{2}{\pi |z|}\right)^{1/2} \cosh(\Im(z))\frac{n^2}{|z|}
\left(\frac{|z|}{|\Re(z)|}\right)^{1/2},
$$
where in the last line we estimated the remainder similarly to the 
terms in the expansion.
\end{proof}
\noindent One proves  explicit asymptotics with one extra term in a similar way: 

\begin{corollary}\label{zlargernsquarefine}
Let $z\in\Co$ be such that $\Re(z)>1/4$ and $\Im(z)<\Re(z)$. Then 
\begin{multline*}
\left|J_{0}^\pm(z)- \left(\frac{2}{\pi z}\right)^{1/2}\cos(\omega_0)
+\frac 1{8z}  \left(\frac{2}{\pi z}\right)^{1/2}\sin(\omega_0)\right|\le \\
\le\frac{9}{128|z|^2}\left(\frac{2}{\pi |z|}\right)^{1/2} \cosh(|\Im (z)|)
\left(\frac{|z|}{|\Re(z)|}\right)^{5/2},
\end{multline*}

\begin{multline*}
\left|J_{1}^\pm(z)- \left(\frac{2}{\pi z}\right)^{1/2}\cos(\omega_1)
+\frac 3{8z}  \left(\frac{2}{\pi z}\right)^{1/2}\sin(\omega_1)\right|\le \\
\le\frac{15}{128|z|^2}\left(\frac{2}{\pi |z|}\right)^{1/2} \cosh(\Im (z))
\left(\frac{|z|}{|\Re(z)|}\right)^{3/2},
\end{multline*}
and if $n>1$ and $\Re(z)>n^2$, then

\begin{multline*}
\left|J_{n}^\pm(z)- \left(\frac{2}{\pi z}\right)^{1/2}\cos(\omega_n)
+\frac {4n^2-1}{8z}  \left(\frac{2}{\pi z}\right)^{1/2}\sin(\omega_n)\right|\le\\ 
\le\frac 14 \left(\frac{2}{\pi |z|}\right)^{1/2} \cosh(\Im(z))\frac{n^4}{|z|^2} \left(\frac{|z|}{|\Re(z)|}\right)^{1/2}.
\end{multline*}

\end{corollary}

\noindent Finally, we need zero order upper bounds for the asymptotic expansion:

\begin{corollary}\label{uniformboundsJ0J1}
For $r>0$, we have that
\begin{itemize}
\item[(a)] $$|J_0(r)|\le  \frac{9}{8}\left(\frac 2{\pi r}\right)^{1/2},$$
\item[(b)] $$|J_1(r)|\le \frac{11}{8} \left(\frac 2{\pi r}\right)^{1/2}.$$
\end{itemize}
\end{corollary}

\begin{proof}[Proof]
This follows for $r>1$ from Corollary \ref{zlargernsquare}, while
for $r\le 1$ it follows from the trivial bound $|J_n|\le 1$.
\end{proof}

\noindent {\it Remark.} Using more refined oscillatory techniques related to Sturm's comparision principle, the sharper bound $r^{1/2}|J_0(r)|\le  \left( 2/{\pi}\right)^{1/2}$ is established in \cite{L}. However, the bounds given by Corollary \ref{uniformboundsJ0J1} suffice for our purpose, and its proof is more in light with the elementary nature of the present paper.


\section{Useful estimates involving the Gamma function}\label{sec:Useful}

A version \cite{Ro} of Stirling's formula
for the Gamma function states that, for $x\ge 0$,

\begin{equation}\label{Stirling}
\Gamma(x)=\sqrt{2 \pi} x^{x-1/2} e^{-x}e^{\mu(x)},
\end{equation}
where the function $\mu$ satisfies the double inequality
\begin{equation}\label{mu_bounds}
\frac 1{12x+1}< \mu(x)< \frac 1{12x}.
\end{equation}
The starting point for all the convex estimates we  need is the following well-known result:

\begin{lemma}\label{GammaConvex}
For $x>0$, the function $x\mapsto \log(\Gamma(x))$ is convex.
\end{lemma}

\begin{proof}
Let $x,y>0$ and $0<\lambda<1$. Set $p=\frac{1}{\lambda}$ and $q=\frac{1}{1-\lambda}$.
It suffices to show that
\begin{equation}\label{logGammaconvex}
\Gamma\Big(\frac{x}{p}+\frac{y}{q}\Big)\leq\Gamma(x)^{\frac{1}{p}}\Gamma(y)^{\frac{1}{q}},
\end{equation}
for then the result follows by taking logarithms on both sides.
To verify \eqref{logGammaconvex}, consider the auxiliary functions 
$$f(t):=t^{\frac{x-1}{p}}e^{-\frac{t}{p}}\textrm{ and }g(t):=t^{\frac{y-1}{q}}e^{-\frac{t}{q}},$$
which satisfy
$$
f(t)g(t)=t^{\frac{x}{p}+\frac{y}{q}-1}e^{-t},
\;\;\; f(t)^p=t^{x-1}e^{-t},\textrm{ and }
 g(t)^q=t^{y-1}e^{-t}.
$$
The result is thus a consequence of H\"older's inequality.
\end{proof}

\begin{corollary}\label{GammaLogConvexityCor}
Let $x\geq\frac 1 2$. Then:
$$\sqrt{x- 1/2} \leq\frac{\Gamma\Big(x+\frac1 2\Big)}{\Gamma(x)}\le \sqrt{x}.$$
\end{corollary}

\begin{proof}
Use the identity $x=\Gamma(x+1)/\Gamma(x)$ together with log convexity of $\Gamma$.   
\end{proof}

\begin{corollary}\label{GammaBeta}
Let $x\geq y>0$ and $w\geq 0$. Then:
$$\frac{\Gamma(x)}{\Gamma(y)}\le \frac{\Gamma(x+w)}{\Gamma(y+w)}.$$ 
\end{corollary}
\begin{proof}
For fixed $w\geq 0$, we claim that the function
$$x\mapsto\frac{\Gamma(x)}{\Gamma(x+w)}$$
is decreasing in $x$. This happens if the inequality
$$\frac{\Gamma'(x)}{\Gamma(x)}\leq \frac{\Gamma'(x+w)}{\Gamma(x+w)}$$
holds for every $x>0$, which in turn is a consequence of the log convexity of $\Gamma$ proved in Lemma \ref{GammaConvex}.
\end{proof}

The following lemma estimates binomial coefficients.

\begin{lemma}\label{GaussianBounds}
Let $x\ge 1$ and let $0\le d< x$ be an integer. Then:

\begin{equation}\label{GaussianBoundsEq}
\frac{\Gamma(2x)}{\Gamma(x-d)\Gamma(x+d)}\le 
\frac{e^{1/24}}{2\sqrt{\pi}} x^{1/2}2^{2x}
e^{-d^2/x}.
\end{equation}
\end{lemma}

\begin{proof}
By Stirling's formula \eqref{Stirling}, we have that
$$\frac{\Gamma(2x)}{\Gamma(x)^2}
=\frac{(2x)^{2x-1/2}}
{\sqrt{2\pi} x^{2x-1}} e^{\mu(2x)-2\mu(x)}
= \frac{x^{1/2} 2^{2x}}
{2 \sqrt{\pi} } e^{\mu(2x)-2\mu(x)}.$$
Also, since $x\geq 1$, we have that $\mu(2x)\leq\frac{1}{24}$, and therefore
$$e^{\mu(2x)-2\mu(x)}\leq e^{\mu(2x)}\leq e^{1/24}.$$
It follows that 

\begin{equation*}
\frac{\Gamma(2x)}{\Gamma(x)^2}
\leq
\frac{e^{1/24}}{2 \sqrt{\pi}} x^{1/2} 2^{2x}.
\end{equation*}
By induction, observe that

\begin{equation*}
\frac{\Gamma(x)\Gamma(x+1)}{\Gamma(x-d)\Gamma(x+d+1)}
=\prod_{j=1}^d\frac{x-j}{x+j}\le \prod_{j=1}^d e^{-2j/x}=e^{-(d^2+d)/x},
\end{equation*}
whenever $d$ is an integer satisfying $0\le d<x$.
This is  a consequence of the elementary inequality
$$\frac{1-y}{1+y}\leq e^{-2y},$$ 
which is valid in particular for $y\in [0,1]$.
Using $(x+d)/d\le e^{d/x}$ we obtain that

\begin{equation*}
\frac{\Gamma(x)\Gamma(x)}{\Gamma(x-d)\Gamma(x+d)}
\leq e^{d/x}\prod_{j=1}^d\frac{x-j}{x+j}\le e^{d/x}\prod_{j=1}^d e^{-2j/x}=e^{-d^2/x},
\end{equation*}
again for integers $d\in [0,x)$. Finally, 

\begin{align*}
\frac{\Gamma(2x)}{\Gamma(x-d)\Gamma(x+d)}=
\frac{\Gamma(2x)}{\Gamma(x)^2}\frac{\Gamma(x)\Gamma(x)}{\Gamma(x-d)\Gamma(x+d)}
\leq
\frac{e^{1/24}}{2 \sqrt{\pi}} x^{1/2} 2^{2x}e^{-d^2/x},
\end{align*}
as desired.
\end{proof}

\noindent{\it Remark.}
The inequality 

$$\frac{\Gamma(x)\Gamma(x)}{\Gamma(x-d)\Gamma(x+d)}\leq e^{-d^2/x},$$ 
 is still valid for non-integer values of $d\in [0,x)$, and therefore Lemma \ref{GaussianBounds} holds in this case as well.
\vspace{.2cm}

 We will also need estimates for the magnitude of the coefficients $a_j(m)$ defined in \eqref{coeffajn} when $j$ is close to $m$. More precisely, in Section \ref{sec:Part2} we will need good upper bounds for the quantities
  $|a_{m+4}(m)|.$
 For small values of $m$, we compute them explicitly.
For large values of $m$, we estimate these quantities via the following lemma.

\begin{lemma}\label{ajnEstimate}
For any natural number $m\geq 1$, we have that

$$|a_{m+4}(m)|
\leq 
\frac{105}{16}\sqrt{\frac 2 \pi} (2m-1)^{-1/2}2^{m}  m^me^{-m}.$$
\end{lemma}

\begin{proof}
The goal is to bound the absolute value of 

$$a_{m+4}(m)=\frac{\Gamma(2m+9/2)}{\Gamma(-7/2)\Gamma(m+5)2^{m+4}}.$$
Making repeated use of the identity $\Gamma(x+1)=x\Gamma(x)$, 
together with the convexity estimate from Corollary \ref{GammaLogConvexityCor} and Stirling's formula \eqref{Stirling}, we have that 

\begin{align*}
|a_{m+4}(m)|&\leq 2^5 \frac{\Gamma(2m-1/2)}{\Gamma(-7/2)\Gamma(m)2^{m+4}}\\
&\leq2(2m-1)^{-1/2} \frac{\Gamma(2m)}{\Gamma(-7/2)\Gamma(m)2^{m}}\\
&= \frac{2(2m-1)^{-1/2} }{\Gamma(-7/2)2^m} \frac{(2m)^{2m-1/2}}{m^{m-1/2}}e^{-m} e^{\mu(2m)-\mu(m)}.
\end{align*}
The bounds \eqref{mu_bounds} for the function $\mu$ imply that $\mu(2m)\leq \mu(m)$, and so $e^{\mu(2m)- \mu(m)}\leq 1$. It follows that

\begin{equation}\label{am+4m}
|a_{m+4}(m)|\leq\frac{2(2m-1)^{-1/2} }{\Gamma(-7/2)2^m} 
 2^{2m-1/2} m^me^{-m}.
 \end{equation}
 The value $\Gamma(-7/2)=\frac{16\sqrt{\pi}}{105}$ can be computed via the second formula in \eqref{HalfIntegerGamma}, and this completes the proof.
 \end{proof}

\section{Part  I. Expanding $J_0$ and $J_1$}\label{sec:Part1}

In the next four sections, we will be working under the standing assumption that $n\geq n_0:=20$.
 We start by asymptotically expanding  the Bessel functions of order 0 and 1 
and their relevant products. 
Due to need of accuracy, we must consider asymptotic expansions of length six and
 keep track of all the terms.
The following notation will be convenient. Let us say that

$$a\sim (a_0)+(a_1)+(a_2)+(a_3)+(a_4)+(a_5) {\rm \ with \ remainders\ }r_0,r_1,r_2,r_3,r_4,r_5,r_6$$ 
if 

$$\Big|a-\sum_{i=1}^k a_{i-1}\Big|\leq r_{k},\textrm{ for every } 0\leq k\leq 6.$$
We also call $r_6$ the last remainder. Suppose additionally that
$$b\sim (b_0)+(b_1)+(b_2)+(b_3)+(b_4)+(b_5) {\rm \ with \ remainders\ }s_0,s_1,s_2,s_3,s_4,s_5,s_6.$$
Then we have the following product formula: 

$$
ab\sim \sum_{k=0}^5 (\sum_{i=0}^k a_i b_{k-i})
$$
with remainders

$$r_0 s_k+\sum_{i=1}^k r_i |b_{k-i}|, \textrm{ for every }0\leq k\leq 6.$$
Recall the coefficients $a_0(n)$ through $a_5(n)$ for $n=0$,

$$1,\ -\frac{1}{8}, \ \frac 9{128},\ -\frac{75}{1024},\ \frac{3675}{32768},
\ -\frac{59535}{262144},
\ \frac{2401245}{4194304},
\ -\frac{57972915}{33554432},
$$
and for $n=1$,

$$1,\ \frac{3}{8},\  -\frac {15}{128},\ \frac{105}{1024},\ -\frac{4725}{32768},\ \frac{72765}{262144},
\ -\frac{2837835}{4194304}
,\ \frac{66891825}{33554432}.
$$

To make the forthcoming notation less cumbersome, let us define, in view of Lemma \ref{AsymptoticBessel},
 
$$\mathfrak{J}_n(r):=\left(\frac{\pi r}{2}\right)^{1/2}J_n(r).$$
To avoid writing many fractions, we further define $t:=(16r)^{-1}$. We also set 
 $$\c:=\cos(r-\pi/4) \textrm{ and } \s:=\sin(r-\pi/4).$$

From Lemma \ref{AsymptoticBessel} and Corollary \ref{uniformboundsJ0J1}, we have that
$$
\mathfrak{J}_0(r)
\sim (\c)+(2t\s)+(-18t^2 \c)+ (-300 t^3 \s)+(7350t^4\c)+(238140 t^5\s)
$$
with remainders discussed there as 

$$
\frac 9 8,\ 2t,\ 
18t^2,\ 
300t^3,\ 
7350 t^4, \ 
238140 t^5, \
9604980 t^6.
$$
In a similar way,
$$
\mathfrak{J}_1(r)
\sim(\s)+ (6 t \c) +(30t^2 \s)+ (-420 t^3 \c)
+(-9450 t^4\s)
+(291060 t^5\c)
$$
with remainders

$$\frac {11} 8 ,\ 
6t,\ 
30t^2,\ 
420 t^3,\
9450t^4, \ 
291060t^5, \
11351340 t^6.
$$
Applying the product formula,
we obtain successively 
\begin{multline*}
\mathfrak{J}_0^2(r)
\sim (\c^2) +(4t \c\s) +(-36t^2\c^2+4 t^2 \s^2) +(-672 t^3 \c\s)+\\
+(15024t^4\c^2-1200t^4\s^2)+(516480 t^5 \c\s)
\end{multline*}
with remainders 

$$\frac{81}{64},\ 
\frac{17}{4}t, \ 
\frac{169}{4}t^2,\
\frac{1419}{2}t^3,\
\frac{68571}{4} t^4, \ 
\frac{1092495}{2} t^5, \
\frac{43435485}{2} t^6,
$$
and 
\begin{align}
\mathfrak{J}_0^3(r) \sim 
\mathfrak{J}_{000}(r):=&
 (\c^3)+(6t\c^2\s)
+(-54 t^2 \c^3+12 t^2 \c\s^2)
+(-1116 t^3\c^2\s+8t^3\s^3)\label{j0third}\\
&+(23022 t^4 \c^3-3816 t^4 \c\s^2)
+(836964 t^5 \c^2\s-3600 t^5\s^3)
\notag
\end{align}
with remainders

$$\frac{729}{512}, \
\frac{217}{32} t,\
\frac{2353}{32}t^2,\
\frac{20003}{16}t^3,\ 
\frac{956787}{32} t^4, \
\frac{15017799}{16} t^5, \ 
\frac{588969477}{16} t^6.
$$
In particular,

\begin{equation}\label{Approx1}
|\mathfrak{J}_0^3(r)-\mathfrak{J}_{000}(r)|
\leq 
\frac{588969477}{16}  (16 r)^{-6}.
\end{equation}
On the other hand, we have 

\begin{multline*}
\mathfrak{J}_1^2(r) \sim (\s^2)+(12t\c\s)+(36t^2\c^2+60 t^2 \s^2)+(-480 t^3 \c\s)+\\
+(-5040 t^4 \c^2 - 18000 t^4 \s^2)
+(443520 t^5 \c \s)
\end{multline*}
with remainders 
$$\frac{121}{64},\ 
\frac{57}{4}t,\  
\frac{429}{4} t^2,\ 
\frac{2715}{2} t^3,\ 
\frac{113535}{4} t^4,\ 
\frac{1659735}{2} t^5,\
\frac{62391105}{2} t^6,
$$
and 

\begin{align}
(\mathfrak{J}_1^2\mathfrak{J}_0)(r)\sim 
\mathfrak{J}_{110}(r):=&(\c\s^2)+(12t\c^2\s+2t\s^3) 
+(36t^2\c^3+66 t^2\c \s^2)+(-624t^3\c^2\s-180 t^3 \s^3)\label{j0j1squared}\\
&+(-5688 t^4 \c^3 - 16290 t^4 \c \s^2)+(519480 t^5 \c^2 \s + 184140 t^5 \s^3)\notag
\end{align}
with remainders 
$$\frac{1089}{512},\ 
\frac{577}{32}t,\  
\frac{5433}{32} t^2,\ 
\frac{38331}{16} t^3,\ 
\frac{1638411}{32} t^4,\ 
\frac{23971455}{16} t^5,\
\frac{897834285}{16} t^6.
$$
In particular,

\begin{equation}\label{Approx2}
|(\mathfrak{J}_1^2\mathfrak{J}_0)(r)-\mathfrak{J}_{110}(r)|
\leq 
\frac{897834285}{16} (16 r)^{-6}.
\end{equation}
Inequalities \eqref{Approx1} and \eqref{Approx2} are at the core of the following result, the proof of which does not require $m$ to be even, nor $m\leq n$.
\vspace{.3cm}

\noindent {\bf Estimate A.}
{\it For $n\geq n_0= 20$ and  $m\geq 0$, we have
\begin{equation}\label{AEstOne}
\Big|\int_0^\infty J_{n+m} J_n \mathfrak{J}_m (\mathfrak{J}_0^3-\mathfrak{J}_{000}) r^{-1} dr\Big|
\leq 0.74 n_0^{-1/2}(n+m)^{-6},
\end{equation}

\begin{equation}\label{AEstTwo}
\Big|\int_0^\infty J_{n+m} J_n \mathfrak{J}_m (\mathfrak{J}_1^2\mathfrak{J}_0-\mathfrak{J}_{110}) r^{-1} dr\Big|
\leq 1.12 n_0^{-1/2}(n+m)^{-6}.
\end{equation}}

\begin{proof}[Proof of Estimate A]
Using the Cauchy-Schwarz inequality and Lemmata \ref{kapteynlemma} and  \ref{WeberSchafheitlin} to compute the integrals, we have that

\begin{align*}
\int_0^\infty |J_n J_{n+m}\mathfrak{J}_m|r^{-7}{dr}
&\leq \Big(\frac{\pi}{2}\Big)^{1/2}
\Big(\int_0^\infty J_n^2\frac{dr}{r}\Big)^{1/2} \Big( \int_0^\infty J_{n+m}^2 \frac{dr}{r^{12}}\Big)^{1/2}\\
&=
\Big(\frac{1}{2n}\Big)^{1/2}\Big(\frac{128}{693}\frac{\Gamma(n+m-\frac{11}{2})}{\Gamma(n+m+\frac{13}{2})}\Big)^{1/2}\\
&\le a_0^{1/2}\Big(\frac{64}{693 n}\Big)^{1/2}(n+m)^{-6}, 
\end{align*}
 where

$$a_0
:=\sup_{\ell\geq n_0} \frac{\Gamma(\ell-\frac{11}2)}{\Gamma(\ell+\frac{13}2)} \ell^{12}
=\frac{\Gamma(\frac{29}2)}{\Gamma(\frac{53}2)}20^{12}
\leq 1.21.$$

\noindent Thus, for $n\ge n_0$, the left-hand side of inequality \eqref{AEstOne} is bounded by

$$\le \frac{588969477/16}{16^6} a_0^{1/2} \Big(\frac{64}{693 n}\Big)^{1/2}(n+m)^{-6} 
\le \alpha n_0^{-1/2}(n+m)^{-6},$$
where 
$$\alpha= \frac{588969477/16}{16^6} a_0^{1/2} \Big(\frac{64}{693 }\Big)^{1/2}\leq 0.74.$$
In a similar way, for $n\ge n_0$, the left-hand side of inequality \eqref{AEstTwo} is bounded by

$$\le \frac{897834285/16}{16^6} a_0^{1/2} \Big(\frac{64}{693 n}\Big)^{1/2}(n+m)^{-6} 
\le \beta n_0^{-1/2}(n+m)^{-6},$$
where 
$$\beta= \frac{897834285/16}{16^6} a_0^{1/2} \Big(\frac{64}{693 }\Big)^{1/2}\leq 1.12.$$
\end{proof}


\section{Part II. Expanding $J_m$}\label{sec:Part2}

Start by noting that $\cos\omega_m=(-1)^{m/2}\c$ and $\sin\omega_m=(-1)^{m/2}\s$ since $m$ is an even integer. 
We have that
\begin{align}\label{asymptJm}
\mathfrak{J}_m(r)-\rho_m(r)
=(-1)^{m/2}\Big(\c\sum_{k=0}^{\frac{m}{2}+1} (-1)^k (16 t)^{2k} a_{2k}(m)
-\s \sum_{k=0}^{\frac{m}{2}+1} (-1)^k(16t)^{2k+1} a_{2k+1}(m)\Big),
\end{align}
where the error term $\rho_m$ is implicitly defined by this identity.
From Lemma \ref{AsymptoticBessel} we know that
\begin{equation}\label{remainderRm}
|\rho_m(r)|\leq \frac{|a_{m+4}(m)|}{r^{m+4}}
\end{equation}
 for every $r\geq 0$.
The main integrals, which are the subject of the next section, arise from replacing 
$\mathfrak{J}_m\mathfrak{J}_{000}$ by 

\begin{equation}\label{DefJm000}
\mathfrak{J}_{m000} := (\mathfrak{J}_m-\rho_m) \mathfrak{J}_{000},
\end{equation}
and $\mathfrak{J}_m\mathfrak{J}_{110}$ by 

\begin{equation}\label{DefJm110}
\mathfrak{J}_{m110} := (\mathfrak{J}_m-\rho_m) \mathfrak{J}_{110}.
\end{equation}

\noindent Estimating the error term in this replacement is the main goal of the present section. We formulate it as
 \vspace{.3cm}
 
\noindent {\bf  Estimate B.}
{\it Let $n\geq n_0= 20$ and  $m$ be even.

If $m=0$, then}
 
\begin{equation}\label{m0B0}
\Big|\int_0^\infty J_{n}^2 (\mathfrak{J}_0\mathfrak{J}_{000}-\mathfrak{J}_{0000}) r^{-1}dr\Big|
\leq
0.022 n_0^{-1} n^{-4},
\end{equation}

\begin{equation}\label{m0B1}
\Big|\int_0^\infty J_{n}^2 (\mathfrak{J}_0\mathfrak{J}_{110}-\mathfrak{J}_{0110}) r^{-1}dr\Big|
\leq
0.023 n_0^{-1} n^{-4},
\end{equation}

{\it If $m=2$, then}

\begin{equation}\label{m2B0}
\Big|\int_0^\infty J_{n+2} J_n (\mathfrak{J}_2\mathfrak{J}_{000}-\mathfrak{J}_{2000}) r^{-1}dr\Big|
\leq
0.162 n_0^{-1}n^{-6},
\end{equation}

\begin{equation}\label{m2B1}
\Big|\int_0^\infty J_{n+2} J_n (\mathfrak{J}_2\mathfrak{J}_{110}-\mathfrak{J}_{2110}) r^{-1}dr\Big|
\leq
0.166 n_0^{-1}n^{-6},
\end{equation}

{\it If $m=4$, then}

\begin{equation*}
\Big|\int_0^\infty J_{n+4} J_n (\mathfrak{J}_4\mathfrak{J}_{000}-\mathfrak{J}_{4000}) r^{-1}dr\Big|
\leq
2.823 n_0^{-1}n^{-8},
\end{equation*}

\begin{equation*}
\Big|\int_0^\infty J_{n+4} J_n (\mathfrak{J}_4\mathfrak{J}_{110}-\mathfrak{J}_{4110}) r^{-1}dr\Big|
\leq
2.885 n_0^{-1}n^{-8},
\end{equation*}

{\it If $6\leq m\leq n$, then}

\begin{equation}\label{m6B0}
\Big|\int_0^\infty J_{n+m} J_n (\mathfrak{J}_m\mathfrak{J}_{000}-\mathfrak{J}_{m000}) r^{-1}dr\Big|
\leq
0.015 n_0^{-1} n^{-4},
\end{equation}

\begin{equation}\label{m6B1}
\Big|\int_0^\infty J_{n+m} J_n (\mathfrak{J}_m\mathfrak{J}_{110}-\mathfrak{J}_{m110}) r^{-1}dr\Big|
\leq
0.015 n_0^{-1} n^{-4}.
\end{equation}

\begin{proof}[Proof of  Estimate B]
We consider the case $m=0$ first. 
\noindent The left-hand side of \eqref{m0B0} is bounded by

\begin{equation}\label{Error1}
|a_{4}(0)| \int_0^\infty J_n^2(r)
(1+6t+66t^2+1124t^3+26838t^4+840564t^5)  r^{-5}dr,
\end{equation}
whereas the left-hand side of \eqref{m0B1} is bounded by

\begin{equation}\label{Error2}
|a_{4}(0)| \int_0^\infty J_n^2(r)
(1+14t+102t^2+804t^3+21978t^4+703620t^5)  r^{-5}dr. 
\end{equation}
For $\ell\in\{5,6,\ldots,10\}$, Lemma \ref{WeberSchafheitlin} implies that

$$\int_0^\infty J_n^2(r) r^{-\ell}dr=\frac{2^{-\ell}\Gamma(\ell)}{\Gamma(\frac{\ell+1}2)^2}\frac{\Gamma(n-\frac{\ell-1}2)}{\Gamma(n+\frac{\ell+1}2)}.$$
This can be estimated as follows:

$$\int_0^\infty J_n^2(r) r^{-\ell}dr\leq
c_\ell^{(0)} n^{-\ell},$$
where 

$$c_\ell^{(0)}:=\frac{2^{-\ell}\Gamma(\ell)}{\Gamma(\frac{\ell+1}2)^2} \frac{\Gamma(20-\frac{\ell-1}2)}{\Gamma(20+\frac{\ell+1}2)} 20^\ell.$$
In particular, for  $\ell\in \{5,6,\ldots,10\}$, one can easily check that

$$0.14\leq c_\ell^{(0)}\leq 0.19.$$
It follows that the upper bounds for each of the last five summands on the right-hand side of inequalities \eqref{Error1} and \eqref{Error2} can be estimated by a small fraction of the upper bound for the first summand. Quantifying this, one obtains
\eqref{m0B0} and \eqref{m0B1}, respectively.

We consider the case $m=2$ next. Again for $\ell\in\{5,6,\ldots,10\}$, the integral to consider is the following:

$$\int_0^\infty |J_n(r)J_{n+2}(r)| r^{-2-\ell}dr.$$
In view of the absolute value in the integrand, this cannot be computed directly with Lemma \ref{WeberSchafheitlin}. Instead, we use the Cauchy-Schwarz inequality to estimate

\begin{align*}
\int_0^\infty |J_nJ_{n+2}| r^{-2-\ell}dr
&\leq
\Big(\int_0^\infty J_n^2 \frac{dr}r\Big)^{1/2}
\Big(\int_0^\infty J_{n+2}^2 \frac{dr}{r^{2\ell+3}}\Big)^{1/2}\\
&= \Big(\frac1{2n}\Big)^{1/2}\Big(\frac{2^{-(2\ell+3)}\Gamma(2\ell+3)}{\Gamma(\ell+2)^2}\frac{\Gamma(n-\ell+1)}{\Gamma(n+\ell+4)}\Big)^{1/2},
\end{align*}
where the last identity is a consequence of  Lemmata \ref{kapteynlemma} and \ref{WeberSchafheitlin}.
Reasoning as before, we derive the estimate

$$\int_0^\infty |J_nJ_{n+2}| r^{-2-\ell}dr
\leq
c_\ell^{(2)} n^{-2-\ell},
$$
where 
$$c_\ell^{(2)}:=\Big(\frac{2^{-(2\ell+3)}\Gamma(2\ell+3)}{2\Gamma(\ell+2)^2} \frac{\Gamma(20-\ell+1)}{\Gamma(20+\ell+4)} 20^{2\ell+3}\Big)^{1/2}.$$
In particular, for every $\ell\in \{5,6,\ldots,10\}$, one checks that

$$0.11\leq c_\ell^{(2)}\leq0.15.$$
As before, it follows that the last five summands on the right-hand sides of estimates

\begin{multline*}
\Big|\int_0^\infty J_{n+2}J_n (\mathfrak{J}_2\mathfrak{J}_{000}-\mathfrak{J}_{2000}) r^{-1}dr\Big|
\leq\\
|a_6(2)|
\Big(c_5^{(2)}n^{-7}+\frac6{16} c_6^{(2)}n^{-8}+\frac{66}{16^2} c_7^{(2)}n^{-9}+\frac{1124}{16^3} c_8^{(2)}n^{-10}+\frac{26838}{16^4} c_9^{(2)}n^{-11}+\frac{840564}{16^5} c_{10}^{(2)}n^{-12}\Big)
\end{multline*}
and 

\begin{multline*}
\Big|\int_0^\infty J_{n+2}J_n (\mathfrak{J}_2\mathfrak{J}_{110}-\mathfrak{J}_{2110}) r^{-1}dr\Big|
\leq\\
|a_6(2)|
\Big(c_5^{(2)}n^{-7}+\frac{14}{16} c_6^{(2)}n^{-8}+\frac{102}{16^2} c_7^{(2)}n^{-9}+\frac{804}{16^3} c_8^{(2)}n^{-10}+\frac{21978}{16^4} c_9^{(2)}n^{-11}+\frac{703620}{16^5} c_{10}^{(2)}n^{-12}\Big)\end{multline*}
can be bounded by a small fraction of the first summand. Quantifying this yields \eqref{m2B0} and \eqref{m2B1}.

The cases $m=4,6,8,10$ can be treated in a completely analogous way to what was done for $m=2$. We omit the details, but remark that for $m=6,8,10$ this method produces estimates which are stronger than \eqref{m6B0} and \eqref{m6B1}. However, the latter will be enough for our purposes.

Finally, we deal with the case of even $m\geq 12$.  Again for $\ell\in\{5,6,\ldots,10\}$, the integral to consider is bounded by
$$\int_0^\infty |J_{n}J_{n+m}| r^{-m-\ell}dr
\leq
\Big(\frac 1{2n}\Big)^{1/2}\Big(\frac{2^{-(2m+2\ell-1)}\Gamma(2m+2\ell-1)}{\Gamma(m+\ell)^2}\frac{\Gamma(n-\ell+1)}{\Gamma(n+2m+\ell)}\Big)^{1/2}.$$
Identifying a binomial coefficient, we notice the trivial bound
 
 $$\frac{2^{-(2m+2\ell-1)}\Gamma(2m+2\ell-1)}{\Gamma(m+\ell)^2}\leq \frac 12.$$
 It follows that
 
 $$\int_0^\infty |J_{n}J_{n+m}| r^{-m-\ell}dr
 \leq
 \frac 12 n^{-1/2}\Big(\frac{\Gamma(n-\ell+1)}{\Gamma(n+2m+\ell)}\Big)^{1/2}.
 $$
To handle the coefficients $|a_{m+4}(m)|$, we recall Lemma \ref{ajnEstimate}. For even $m\geq 12$, define 

$$c^{(m)}_\ell:=\frac1 2\sqrt{\frac 2\pi} (2m-1)^{-1/2} 2^m e^{-m} \Big(\frac {20^{2\ell-1}}{\prod_{k=0}^{2\ell-2} (20-\ell+1+k)}\Big)^{1/2},$$
a decreasing function of $m$ for fixed $\ell$.
We finally arrive at 

\begin{multline*}
\Big|\int_0^\infty J_{n+m}J_n (\mathfrak{J}_m\mathfrak{J}_{000}-\mathfrak{J}_{m000}) r^{-1}dr\Big|
\leq\\
\frac{105}{16}\Big(c^{(m)}_{5}n^{-5}+\frac 6{16}c^{(m)}_{6}n^{-6}+\frac{66}{16^2}c^{(m)}_{7}n^{-7}+\frac{1124}{16^3}c^{(m)}_{8}n^{-8}+\frac{26828}{16^4}c^{(m)}_{9}n^{-9}+\frac{840564}{16^5}c^{(m)}_{10}n^{-10}\Big)
\end{multline*}
and 

\begin{multline*}
\Big|\int_0^\infty J_{n+m}J_n (\mathfrak{J}_m\mathfrak{J}_{110}-\mathfrak{J}_{m110}) r^{-1}dr\Big|
\leq\\
\frac{105}{16}\Big(c^{(m)}_{5}n^{-5}+\frac {14}{16}c^{(m)}_{6}n^{-6}+\frac{102}{16^2}c^{(m)}_{7}n^{-7}+\frac{804}{16^3}c^{(m)}_{8}n^{-8}+\frac{21978}{16^4}c^{(m)}_{9}n^{-9}+\frac{703620}{16^5}c^{(m)}_{10}n^{-10}\Big).
\end{multline*}
If $m\geq 12$, then both of these expressions are bounded by 
$ 0.015 n_0^{-1} n^{-4}$, as claimed.
\end{proof}


\section{Part III. The core integrals}\label{sec:Part3}

The main integrals that are left to analyze,

$$\mu_0=\mu_0(m,n):=\int_0^\infty J_n(r)J_{n+m}(r)\mathfrak{J}_{m000}(r) r^{-1} dr;$$
$$\mu_1=\mu_1(m,n):=\int_0^\infty J_n(r)J_{n+m}(r)\mathfrak{J}_{m110}(r) r^{-1} dr,$$
decompose into the core integrals \eqref{mainintegrals} by means
of expanding $\mathfrak{J}_{m000}$ and $\mathfrak{J}_{m110}$ using the following elementary trigonometric
facts:
\begin{align*}
&8\c^{4}=-\cos(4r)+4\sin(2r)+3, \\
&8\c^{3}\s=-\sin(4r)-2\cos(2r), \\
&8\c^{2}\s^2=\cos(4r)+1, \\
&8\c\s^3=\sin(4r)-2\cos(2r), \\
&8\s^{4}=-\cos(4r)-4\sin(2r)+3. 
\end{align*}
These identities can be readily checked recalling the definitions  $\c=\cos(r-\pi/4)$ and $\s=\sin(r-\pi/4)$, and noting again that $\cos\omega_m=(-1)^{m/2}\c$ and 
$\sin\omega_m=(-1)^{m/2}\s$.

 A simple parity check
verifies that the resulting core integrals with 
$\cos(2r)$ and $\sin(2r)$ satisfy the parity assumption
of Lemma \ref{2rlemma} relative to the powers of
$r$, and so these terms yield zero contribution. 
It therefore suffices  to consider the constant terms
and the terms involving $\cos(4r)$ and $\sin(4r)$.
The strategy will be to split the main integrals

\begin{align*}
\mu_0&=\mu_0^{(\cos)}+\mu_0^{(\sin)}; \\ 
\mu_1&=\mu_1^{(\cos)}+\mu_1^{(\sin)},
\end{align*}
 according to cosine and sine contributions. 
More precisely, recall definitions \eqref{DefJm000} and \eqref{DefJm110} for $\mathfrak{J}_{m000} $ and $\mathfrak{J}_{m110} $, respectively. The first factor in each of them, namely $\mathfrak{J}_m-\rho_m$, is given by identity \eqref{asymptJm}. The right-hand side of this identity consists of two sums which come affected by a coefficient of  $\mathfrak{c}$ and $\mathfrak{s}$. These are at the source of what we denote by cosine and sine contributions, respectively. Working out the algebra, one is led to define

\begin{multline}\label{MuCos}
\mu_*^{(\cos)}:=\frac{(-1)^{m/2}}{8}\sum_{k=0}^{\frac{m}{2}+1} (-1)^k a_{2k}(m)
\int_0^\infty J_n(r)J_{n+m}(r)
\cdot\Big[ (\alpha_0 +\alpha_2 t^2 +\alpha_4 t^4) +\\
+(\beta_0+\beta_2 t^2+\beta_4 t^4)\cos(4 r) +(\gamma_1 t+\gamma_3 t^3+\gamma_5 t^5)\sin(4 r)\Big] r^{-2k-1}  dr
\end{multline}
for the sequence of coefficients given by

\begin{equation}
(\alpha_0,\alpha_2,\alpha_4)= \begin{cases}
    (3,-150,65250), & \textrm{if $*=0$},\\
   (1,174,-33354), & \textrm{if $*=1$};
  \end{cases} \label{eq:alphasCos}
\end{equation}

\begin{equation}
(\beta_0,\gamma_1,\beta_2,\gamma_3,\beta_4,\gamma_5)= \begin{cases}
   (-1,-6,66,1124,-26838,-840564), & \textrm{if $*=0$},\\
   (1,-10,30,444,-10602,-335340), & \textrm{if $*=1$},
  \end{cases} \label{betasgammasCos}
\end{equation}
 and 

\begin{multline}\label{MuSin}
\mu_*^{(\sin)}:=-\frac{(-1)^{m/2}}{8}\sum_{k=0}^{\frac{m}{2}+1} (-1)^k a_{2k+1}(m)
\int_0^\infty J_n(r)J_{n+m}(r)
\cdot\Big[ (\alpha_1 t+\alpha_3 t^3 +\alpha_5 t^5) +\\
+(\beta_1 t+\beta_3 t^3+\beta_5 t^5)\cos(4 r) +(\gamma_0 +\gamma_2 t^2+\gamma_4 t^4)\sin(4 r)\Big] r^{-2k-2}  dr
\end{multline} 
 for the sequence of  coefficients  
  given by 
 
\begin{equation}
(\alpha_1,\alpha_3,\alpha_5)= \begin{cases}
   (6,-1092,826164), & \textrm{if $*=0$},\\
   (18,-1164,1071900), & \textrm{if $*=1$};
  \end{cases} \label{eq:alphasSin}
\end{equation}

\begin{equation}
(\gamma_0,\beta_1,\gamma_2,\beta_3,\gamma_4,\beta_5)= \begin{cases}
   (-1,6,66,-1124,-26838,840564), & \textrm{if $*=0$},\\
   (1,10,30,-444,-10602,335340), & \textrm{if $*=1$}.
  \end{cases} \label{betasgammasSin}
\end{equation}
For $*\in\{0,1\}$, the goal is to obtain a set of estimates of the form

\begin{align}
|\mu_*^{(\cos)}(m,n)-M_*^{(\cos)}(m,n)|&\leq E_{1,*}^{(\cos)}(m,n)+E_{2,*}^{(\cos)}(m,n);\label{EstMuCos}\\
|\mu_*^{(\sin)}(m,n)-M_*^{(\sin)}(m,n)|&\leq E_{1,*}^{(\sin)}(m,n)+E_{2,*}^{(\sin)}(m,n),\label{EstMuSin}
\end{align}
where $M$ and $E_1$ denote the main and error terms coming from the analysis of the constant terms, and $E_2$ denotes the error term coming from the analysis of the terms of frequency $4r$. We shall denote $E_1$ and $E_2$ by error terms of the first and second kind, respectively.

\subsection{Constant terms}
Everything originating from the constant terms can be explicitly computed, one just needs to be careful about bookkeeping. As indicated before, we organize the terms into cosine and sine contributions. 
\subsubsection{Cosine contributions}

We split the analysis in four cases: $m=0$, $m=2$, $m=4$ and $m\geq 6$.
In each of these four cases we will identify, as announced, a main term $M$ and an error term $E_1$.

Let us start by handling the case $m=0$. In this case, the  contribution coming from the non-oscillatory term $\alpha_0+\alpha_2 t^2+\alpha_4 t^4$ in \eqref{MuCos} can by computed exactly with the help of Lemmata \ref{kapteynlemma} and \ref{WeberSchafheitlin}, the result being a main term 

$$M_*^{(\cos)}(0,n):=\frac{1}{8}\alpha_0\frac{1}{2n}+\frac 1 8\Big(a_0(0)\frac{\alpha_2}{16^2}-a_2(0)\alpha_0\Big)\frac{1}{4(n-1)n(n+1)},$$
and an error term of the first kind

$$
E_{1,*}^{(\cos)}(0,n):=\frac 1 8\Big(a_0(0)\frac{\alpha_4}{16^4}-a_2(0)\frac{\alpha_2}{16^2}\Big)\frac{3\Gamma(n-2)}{16 \Gamma(n+3)}
+\frac 1 8\Big(-a_2(0)\frac{\alpha_4}{16^4}\Big)\frac{5\Gamma(n-3)}{32\Gamma(n+4)}.
$$
 Recalling \eqref{eq:alphasCos}, the main term can be computed as follows:
 \begin{equation}\label{MainCos0}
M_{*}^{(\cos)}(0,n)= \begin{cases}
   \frac 3{16n}-\frac{51}{2048 (n-1)n(n+1)}, & \textrm{if $*=0$},\\
   \frac 1{16n}+\frac{39}{2048 (n-1)n(n+1)}, & \textrm{if $*=1$}.
  \end{cases} 
\end{equation}
 To estimate the error term, we first compute it as

\begin{equation*}
E_{1,*}^{(\cos)}(0,n)= \begin{cases}
    \frac{101925\Gamma(n-2)}{4194304\Gamma(n+3)}-\frac{1468125\Gamma(n-3)}{1073741824\Gamma(n+4)} , & \textrm{if $*=0$},\\
    -\frac{54729\Gamma(n-2)}{4194304\Gamma(n+3)}+\frac{750465\Gamma(n-3)}{1073741824\Gamma(n+4)}, & \textrm{if $*=1$}.
  \end{cases} 
\end{equation*}
Using the triangle inequality together with the easily verified bounds
$$\frac{1}{n^5}\leq \frac{\Gamma(n-2)}{\Gamma(n+3)} \leq\frac{1.02}{n^5} 
\textrm{ and }
\frac{1}{n^7}\leq \frac{\Gamma(n-3)}{\Gamma(n+4)} \leq\frac{1.04}{n^7},$$
valid for $n\geq 20$, one arrives at

\begin{equation}\label{E1Cos0}
|E_{1,*}^{(\cos)}(0,n)|\leq \begin{cases}
    0.026 n_0^{-1} n^{-4}, & \textrm{if $*=0$},\\
    0.015 n_0^{-1} n^{-4}, & \textrm{if $*=1$}.
  \end{cases} 
\end{equation}

We move on to the case $m=2$. Orthogonality kicks in the form of Lemma \ref{WeberSchafheitlin} to ensure that we only have one main term, which the same lemma  computes as 

$$M_*^{(\cos)}(2,n):=\frac{1}{8}\Big(-a_0(2)\frac{\alpha_2}{16^2}+a_2(2)\alpha_0\Big)\frac{1}{8n(n+1)(n+2)}.$$
In other words,
 \begin{equation}\label{MainCos2}
M_{*}^{(\cos)}(2,n)= \begin{cases}
   \frac{195}{4096 n(n+1)(n+2)}, & \textrm{if $*=0$},\\
   \frac{9}{4096 n(n+1)(n+2)}, & \textrm{if $*=1$}.
  \end{cases} 
\end{equation}
The error term of the first kind is now given by

\begin{align*}
E_{1,*}^{(\cos)}(2,n):=
&\frac{1}{8}\Big(-a_0(2)\frac{\alpha_4}{16^4}+a_2(2)\frac{\alpha_2}{16^2}-a_4(2){\alpha_0}\Big)\frac{\Gamma(n-1)}{8\Gamma(n+4)}\\
&+\frac{1}{8}\Big(a_2(2)\frac{\alpha_4}{16^4}-a_4(2)\frac{\alpha_2}{16^2}\Big)\frac{15\Gamma(n-2)}{128\Gamma(n+5)}+\frac 1 8 \Big(-a_4(2)\frac{\alpha_4}{16^4}\Big)\frac{7\Gamma(n-3)}{64\Gamma(n+6)}.
\end{align*}
Proceeding as in the case $m=0$, we see that this term obeys the following estimate:

\begin{equation}\label{E1Cos2}
|E_{1,*}^{(\cos)}(2,n)|\leq \begin{cases}
    0.039 n_0^{-1} n^{-4}, & \textrm{if $*=0$},\\
    0.012 n_0^{-1} n^{-4}, & \textrm{if $*=1$}.
  \end{cases} 
\end{equation}

In the case $m=4$, we expand to one higher order. The reason for this will become apparent at the end of Section \ref{sec:numerical}. We thus have exactly one main term

$$M_*^{(\cos)}(4,n):=\frac{1}{8}\Big(a_{0}(4)\frac{\alpha_4}{16^4}-a_{2}(4)\frac{\alpha_2}{16^2}+a_4(4)\alpha_0\Big)\frac{1}{32 n(n+1)(n+2)(n+3)(n+4)},$$
and an error term of the first kind

\begin{align*}
E_{1,*}^{(\cos)}(4,n):=
&\frac{1}{8}\Big(-a_{2}(4)\frac{\alpha_4}{16^4}+a_4(4)\frac{\alpha_2}{16^2}-a_{6}(4)\alpha_0\Big)\frac{3\Gamma(n-1)}{64 \Gamma(n+6)}\\
&+\frac{1}{8}\Big(a_4(4)\frac{\alpha_4}{16^4}-a_{6}(4)\frac{\alpha_2}{16^2}\Big)\frac{7\Gamma(n-2)}{128 \Gamma(n+7)}+\frac{1}{8}\Big(-a_{6}(4)\frac{\alpha_4}{16^4}\Big)\frac{15\Gamma(n-3)}{256 \Gamma(n+8)}.
\end{align*} 
As before, we compute the main term
 
 \begin{equation}\label{MainCos4}
M_{*}^{(\cos)}(4,n)= \begin{cases}
   \frac{322425}{1048576 n(n+1)(n+2)(n+3)(n+4)}, & \textrm{if $*=0$},\\
   \frac{7011}{1048576 n(n+1)(n+2)(n+3)(n+4)}, & \textrm{if $*=1$},
  \end{cases} 
\end{equation}
and verify the following bounds for the error term:

\begin{equation}\label{E1Cos4}
|E_{1,*}^{(\cos)}(4,n)|\leq \begin{cases}
    0.42 n_0^{-1} n^{-6}, & \textrm{if $*=0$},\\
    0.11 n_0^{-1} n^{-6}, & \textrm{if $*=1$}.
  \end{cases} 
\end{equation}

If $m\geq 6$, orthogonality ensures that there is no main term. The error term of the first kind is  given by
\begin{align*}
E_{1,*}^{(\cos)}(m,n):=
&\frac{1}{8}\Big(a_{m-4}(m)\frac{\alpha_4}{16^4}-a_{m-2}(m)\frac{\alpha_2}{16^2}+a_m(m)\alpha_0\Big)\frac{\Gamma(n)}{2^{m+1} \Gamma(n+m+1)}\\
&+\frac{1}{8}\Big(-a_{m-2}(m)\frac{\alpha_4}{16^4}+a_m(m)\frac{\alpha_2}{16^2}-a_{m+2}(m)\alpha_0\Big)\frac{(m+2)\Gamma(n-1)}{2^{m+3} \Gamma(n+m+2)}\\
&+\frac{1}{8}\Big(a_m(m)\frac{\alpha_4}{16^4}-a_{m+2}(m)\frac{\alpha_2}{16^2}\Big)\frac{(m+3)(m+4)\Gamma(n-2)}{2^{m+6} \Gamma(n+m+3)}\\
&+\frac{1}{8}\Big(-a_{m+2}(m)\frac{\alpha_4}{16^4}\Big)\frac{(m+4)(m+5)(m+6)\Gamma(n-3)}{3\cdot 2^{m+8} \Gamma(n+m+4)},
\end{align*} 
and this can be crudely bounded in the following way.  If $m\geq 6$, then

\begin{equation}\label{monotonicity}
|E_{1,*}^{(\cos)}(m,n)|\leq |E_{1,*}^{(\cos)}(6,n)|
\end{equation}
 for the given range of admissible $m$ and $n$.
It is easy to see that inequality \eqref{monotonicity} holds if $m$ is large enough, essentially because each of the summands that constitute the left-hand side of that inequality is of order at most $n^{-(m+1)}$. For the remaining cases, one checks it directly. The upshot is a bound of the form 

\begin{equation}\label{E1Cosm}
|E_{1,*}^{(\cos)}(m,n)|\leq \begin{cases}
    6.34 n_0^{-3} n^{-4}, & \textrm{if $*=0$},\\
    0.09 n_0^{-3} n^{-4}, & \textrm{if $*=1$},
  \end{cases} 
\end{equation}
valid for every even $m\geq 6$.

\subsubsection{Sine contributions}
We proceed similarly, again splitting the analysis into four cases. 

If $m=0$, then the contribution coming from $\alpha_1 t+\alpha_3 t^3+\alpha_5 t^5$ amounts to a main term
$$M_*^{(\sin)}(0,n):=-\frac{1}{8}a_1(0)\frac{\alpha_1}{16}\frac{1}{4(n-1)n(n+1)},$$
and an error term of the first kind

\begin{align*}
-E_{1,*}^{(\sin)}(0,n):=
&\frac{1}{8}\Big(a_{1}(0)\frac{\alpha_3}{16^3}-a_{3}(0)\frac{\alpha_1}{16}\Big)\frac{3\Gamma(n-2)}{16\Gamma(n+3)}\\
&+\frac{1}{8}\Big(a_{1}(0)\frac{\alpha_5}{16^5}-a_3(0)\frac{\alpha_3}{16^3}\Big)\frac{5\Gamma(n-3)}{32\Gamma(n+4)}+\frac{1}{8}\Big(a_3(0)\frac{\alpha_5}{16^5}\Big)\frac{35\Gamma(n-4)}{256\Gamma(n+5)}.
\end{align*} 
Recalling \eqref{eq:alphasSin}, we compute the main term as

 \begin{equation}\label{MainSin0}
M_{*}^{(\sin)}(0,n)= \begin{cases}
   \frac{3}{ 2048(n-1)n(n+1)}, & \textrm{if $*=0$},\\
   \frac{9}{ 2048(n-1)n(n+1)}, & \textrm{if $*=1$}.
  \end{cases} 
\end{equation}
Arguing as in the last subsection, the error term can be seen to obey the following bounds:
\begin{equation}\label{E1Sin0}
|E_{1,*}^{(\sin)}(0,n)|\leq \begin{cases}
    0.0016 n_0^{-1} n^{-4}, & \textrm{if $*=0$},\\
    0.0030 n_0^{-1} n^{-4}, & \textrm{if $*=1$}.
  \end{cases} 
\end{equation}

If $m=2$, there is a main term

 \begin{equation}\label{MainSin2}
M_*^{(\sin)}(2,n):=\frac{1}{8}\Big(a_1(2)\frac{\alpha_1}{16}\Big)\frac{1}{8n(n+1)(n+2)}= \begin{cases}
   \frac{45}{ 4096 n(n+1)(n+2)}, & \textrm{if $*=0$},\\
   \frac{135}{ 4096 n(n+1)(n+2)}, & \textrm{if $*=1$},
  \end{cases} 
\end{equation}
and an error term of the first kind

\begin{align*}
-E_{1,*}^{(\sin)}(2,n):=
&\frac{1}{8}\Big(-a_{1}(2)\frac{\alpha_3}{16^3}+a_{3}(2)\frac{\alpha_1}{16}\Big)\frac{\Gamma(n-1)}{8\Gamma(n+4)}\\
&+\frac{1}{8}\Big(-a_{1}(2)\frac{\alpha_5}{16^5}-a_3(2)\frac{\alpha_3}{16^3}-a_5(2)\frac{\alpha_1}{16}\Big)\frac{15\Gamma(n-2)}{128\Gamma(n+5)}\\
&+\frac{1}{8}\Big(a_3(2)\frac{\alpha_5}{16^5}-a_5(2)\frac{\alpha_3}{16^3}\Big)\frac{7\Gamma(n-3)}{64\Gamma(n+6)}+\frac 1 8 \Big(-a_5(2)\frac{\alpha_5}{16^5}\Big)\frac{105\Gamma(n-4)}{1024\Gamma(n+7)}
\end{align*} 
which satisfies
\begin{equation}\label{E1Sin2}
|E_{1,*}^{(\sin)}(2,n)|\leq \begin{cases}
    0.0062 n_0^{-1} n^{-4}, & \textrm{if $*=0$},\\
    0.0031 n_0^{-1} n^{-4}, & \textrm{if $*=1$}.
  \end{cases} 
\end{equation}

If $m=4$, we again expand to one higher order. There is a main term

$$M_*^{(\sin)}(4,n):=-\frac{1}{8}\Big(a_{1}(4)\frac{\alpha_3}{16^3}-a_{3}(4)\frac{\alpha_1}{16}\Big)\frac{1}{32 n(n+1)(n+2)(n+3)(n+4)}$$

and an error term of the first kind

\begin{align*}
-E_{1,*}^{(\sin)}(4,n):=
&\frac{1}{8}\Big(a_{1}(4)\frac{\alpha_5}{16^5}-a_{3}(4)\frac{\alpha_3}{16^3}+a_{5}(4)\frac{\alpha_1}{16}\Big)\frac{3\Gamma(n-1)}{64 \Gamma(n+6)}\\
&+\frac{1}{8}\Big(-a_{3}(4)\frac{\alpha_5}{16^5}+a_{5}(4)\frac{\alpha_3}{16^3}-a_{7}(4)\frac{\alpha_1}{16}\Big)\frac{7\Gamma(n-2)}{128 \Gamma(n+7)}\\
&+\frac{1}{8}\Big(a_{5}(4)\frac{\alpha_5}{16^5}-a_{7}(4)\frac{\alpha_3}{16^3}\Big)\frac{15\Gamma(n-3)}{256 \Gamma(n+8)}+\frac 1 8\Big(-a_{7}(4)\frac{\alpha_5}{16^5}\Big)\frac{495\Gamma(n-4)}{8192\Gamma(n+9)}.
\end{align*} 
The main term satisfies

 \begin{equation}\label{MainSin4}
M_*^{(\sin)}(4,n):= \begin{cases}
   \frac{76167}{ 1048576 n(n+1)(n+2)(n+3)(n+4)}, & \textrm{if $*=0$},\\
   \frac{211869}{ 1048576 n(n+1)(n+2)(n+3)(n+4)}, & \textrm{if $*=1$},
  \end{cases} 
\end{equation}
and the error term can be bounded as follows:

\begin{equation}\label{E1Sin4}
|E_{1,*}^{(\sin)}(4,n)|\leq \begin{cases}
    0.086 n_0^{-1} n^{-6}, & \textrm{if $*=0$},\\
    0.063 n_0^{-1} n^{-6}, & \textrm{if $*=1$}.
  \end{cases} 
\end{equation}

Finally, if $m\geq 6$, there is no main term, and the error term of the first kind is given by

\begin{align*}
-E_{1,*}^{(\sin)}&(m,n):=
\frac{1}{8}\Big(-a_{m-5}(m)\frac{\alpha_5}{16^5}+a_{m-3}(m)\frac{\alpha_3}{16^3}-a_{m-1}(m)\frac{\alpha_1}{16}\Big)\frac{\Gamma(n)}{2^{m+1} \Gamma(n+m+1)}\\
&+\frac{1}{8}\Big(a_{m-3}(m)\frac{\alpha_5}{16^5}-a_{m-1}(m)\frac{\alpha_3}{16^3}+a_{m+1}(m)\frac{\alpha_1}{16}\Big)\frac{(m+2)\Gamma(n-1)}{2^{m+3} \Gamma(n+m+2)}\\
&+\frac{1}{8}\Big(-a_{m-1}(m)\frac{\alpha_5}{16^5}+a_{m+1}(m)\frac{\alpha_3}{16^3}-a_{m+3}(m)\frac{\alpha_1}{16}\Big)\frac{(m+3)(m+4)\Gamma(n-2)}{2^{m+6} \Gamma(n+m+3)}\\
&+\frac{1}{8}\Big(a_{m+1}(m)\frac{\alpha_5}{16^5}-a_{m+3}(m)\frac{\alpha_3}{16^3}\Big)\frac{(m+4)(m+5)(m+6)\Gamma(n-3)}{3\cdot 2^{m+8} \Gamma(n+m+4)}\\
&+\frac 1 8\Big(-a_{m+3}(m)\frac{\alpha_5}{16^5}\Big)\frac{(m+5)(m+6)(m+7)(m+8)\Gamma(n-4)}{3\cdot 2^{m+12}\Gamma(n+m+5)}.
\end{align*} 
Again the monotonicity formula

$$|E_{1,*}^{(\sin)}(m,n)|\leq |E_{1,*}^{(\sin)}(6,n)|$$
holds for every even $m\geq 6$, and this implies a bound of the form
\begin{equation}\label{E1Sinm}
|E_{1,*}^{(\sin)}(m,n)|\leq \begin{cases}
    1.49 n_0^{-3} n^{-4}, & \textrm{if $*=0$},\\
    4.08 n_0^{-3} n^{-4}, & \textrm{if $*=1$},
  \end{cases} 
\end{equation}
which is valid in that range of $m$.

\subsection{Frequency $4r$ terms}
To handle the terms of frequency $4r$, we make repeated use of the following result:

\begin{proposition}\label{Careful4rEstimate}
Let $n,m\in\N$ be such that $n\geq n_0= 20$ and $m$ even with $0\leq m\leq n$. Let $\alpha\in \{1,3,5\}$ and $\beta\in\{2,4,6\}$. Then each of the following quantities is less than $ n^{-1} 0.35^{n}$:

\begin{align*}
(i)  \quad &\sum_{k=0}^{\frac{m}{2}+1}  \left|\int_0^\infty J_n(r)J_{n+m}(r) 
r^{-2k}
a_{2k}(m)\cos(4r) r^{-\alpha} \, dr\right|,
\\
(ii) \quad  &\sum_{k=0}^{\frac{m}{2}+1}  \left|\int_0^\infty J_n(r)J_{n+m}(r) 
r^{-2k}
a_{2k}(m)\sin(4r) r^{-\beta} \, dr\right|,
\\
(iii) \quad &\sum_{k=0}^{\frac{m}{2}+1}  \left|\int_0^\infty J_n(r)J_{n+m}(r) 
r^{-2k-1}
a_{2k+1}(m)\cos(4r) r^{-\beta} \, dr\right|,
\\
(iv) \quad  &\sum_{k=0}^{\frac{m}{2}+1}  \left|\int_0^\infty J_n(r)J_{n+m}(r) 
r^{-2k-1}
a_{2k+1}(m)\sin(4r) r^{-\alpha} \, dr\right|.
\end{align*}
\end{proposition}

\begin{proof}
All estimates can be proved in a very similar way. We focus on the tightest case, that of $(i)$ with $\alpha=1$, and briefly comment on the other cases at the end of the proof.
Using the definition of the coefficients $a_j(n)$ and Lemma \ref{L2}, together with the convexity estimate from Corollary \ref{GammaBeta}, we obtain

\begin{align}
\sum_{k=0}^{\frac{m}{2}+1} |a_{2k}(m)| \Big|\int_0^\infty J_n&(r)J_{n+m}(r) 
r^{-2k}
\cos(4r) r^{-1} \, dr\Big| 
\leq\notag\\
&\leq \sum_{k=0}^{\frac{m}{2}+1}  \frac{\Gamma(m+2k+1/2)}{\Gamma(m-2k+1/2)2^{2k}\Gamma(2k+1)}
\frac{2^{2k}}{4^{2n+m}}\frac{\Gamma(2n+m-2k)}{\Gamma(n+1)\Gamma(n+m+1)}\notag\\
&\leq \sum_{k=0}^{\frac{m}{2}+1}  \frac{\Gamma(m+2k+1/2)}{\Gamma(m-2k+1/2)\Gamma(4k+1)}
\frac{\Gamma(2n+m)}{\Gamma(n+1)\Gamma(n+m+1)}4^{-2n-m}.\label{TwoFractionsBigSum}
\end{align}
The second fraction in this expression resembles a binomial coefficient and does not depend on $k$. It can be estimated in the following way:
 using Lemma \ref{GaussianBounds} with $x=n+m/2$ and $d=m/2$, we see that
\begin{align}
\frac{\Gamma(2n+m)}{\Gamma(n+1)\Gamma(n+m+1)}
&=\frac{1}{n(n+m)}\frac{\Gamma(2n+m)}{\Gamma(n)\Gamma(n+m)}\notag\\
&\leq \frac{1}{n(n+m)} \frac{e^{1/24}}{2\sqrt{\pi}} (n+m/2)^{1/2} 2^{2n+m} e^{-\frac{(m/2)^2}{n+m/2}}\notag\\
&\leq\frac{e^{1/24}}{2\sqrt{\pi}}\frac{1}{n(n+m)^{1/2}}  2^{2n+m}.\label{nokBin}
\end{align}
To estimate the sum of the first fractions in \eqref{TwoFractionsBigSum}, we proceed as follows. 
For $k=m/2+1$, we simply have that

\begin{multline*}
\frac{\Gamma(m+2k+1/2)}{\Gamma(m-2k+1/2)\Gamma(4k+1)}=
\frac{\Gamma(2m+5/2)}{\Gamma(-3/2)\Gamma(2m+5)}\leq\\
\leq\frac{1}{\Gamma(-3/2)}\frac{\Gamma(2m+3)}{\Gamma(2m+5)}
=\frac{3}{4\sqrt{\pi}}\frac{1}{(2m+4)(2m+3)}.
\end{multline*} 
On the other hand, as long as\footnote{This  does not work for $k=m/2+1$ because the assumptions of Corollary \ref{GammaBeta} are not met and the conclusion fails.} $0\leq k\leq m/2$, we can use  Corollary \ref{GammaBeta} followed by Lemma \ref{GaussianBounds} with $x=m/2+k+1$ and $d=|3k-m/2|$ to conclude that:
 
 \begin{align*}
 \sum_{k=0}^{\frac{m}{2}}  \frac{\Gamma(m+2k+1/2)}{\Gamma(m-2k+1/2)\Gamma(4k+1)}
 &\leq
 \sum_{k=0}^{\frac{m}{2}}  \frac{\Gamma(m+2k+1)}{\Gamma(m-2k+1)\Gamma(4k+1)}\\
&= \sum_{k=0}^{\frac{m}{2}} \frac{\Gamma(m+2k+2)}{\Gamma(m-2k+1)\Gamma(4k+1)}\frac{1}{m+2k+1} \\
 &\leq \frac{2 e^{1/24}}{\sqrt{\pi}} 2^m \sum_{k=0}^{\frac{m}{2}} \frac{(m/2+k+1)^{1/2}}{m+2k+1}4^{k} e^{-\frac{(3k-m/2)^2}{m/2+k+1}}.
 \end{align*}
For $0\leq k\leq m/2$, it is easy to check that the quantity
$\frac{(m/2+k+1)^{1/2}}{m+2k+1}$
is decreasing in $k$. Moreover, for $k=0$, we have that

$$\frac{(m/2+1)^{1/2}}{m+1}\leq (m+1)^{-1/2}.$$
Using this, we are left to estimate the Gaussian sum

$$\Upsilon_m:=\sum_{k=0}^{\frac{m}{2}} 4^{k} e^{-\frac{(3k-m/2)^2}{m/2+k+1}}.$$
We start with the trivial estimate
$$\Upsilon_m\leq \sum_{k=0}^{\frac{m}{2}} 4^{k} e^{-\frac{(3k-m/2)^2}{m+1}}.$$
Changing variables of summation $3\ell=3k-m/2$, we see that

$$\Upsilon_m\leq 2^{m/3}\sum_{\ell\in L} 4^\ell  e^{-\frac{(3\ell)^2}{m+1}},$$
where $L$ is the new summation set, given by
$$L:=\Big\{-\frac m 6,-\frac m 6+1,\ldots,\frac m 3\Big\}\subset \frac{1}{3}\Z.$$
We estimate this sum by the product of the largest term and the number of terms $\# L={m}/{2}+1$. To detect the largest term, define the function

$$\varphi_m(x):=4^x  e^{-\frac{(3x)^2}{m+1}}.$$
The unique solution $x_0\in [-\frac m 6, \frac m 3]$ to the stationary condition $\varphi_m'(x_0)=0$ is given by
$$x_0=\frac{\log 2}{9} (m+1),$$
for which we have 
$$\varphi_m(x_0)=A^{m+1},$$
where
$$A:={4^{\frac{\log 2}{9}}e^{-(\frac{\log 2}{3})^2}}\leq 1.06.$$
Since
$\varphi_m(\ell)\leq \varphi_m(x_0)$ for every $\ell\in L$,  we thus have that

$$\Upsilon_m\leq \Big( \frac{m}{2}+1 \Big) 2^{m/3} A^{m+1}.$$
It follows that

\begin{align} 
\sum_{k=0}^{\frac{m}{2}+1}  &\frac{\Gamma(m+2k+1/2)}{\Gamma(m-2k+1/2)\Gamma(4k+1)}
\leq\label{kSum}\\
&\leq
 \frac{2 e^{1/24}}{\sqrt{\pi}} (m+1)^{-1/2}2^m \Big( \frac{m}{2}+1 \Big) 2^{m/3} A^{m+1}
 +\frac{3}{4\sqrt{\pi}}\frac{1}{(2m+4)(2m+3)}\notag\\
&\leq
\frac{103}{100} \frac{2 e^{1/24}}{\sqrt{\pi}} (m+1)^{-1/2}2^m \Big( \frac{m}{2}+1 \Big) 2^{m/3} A^{m+1},\notag
 \end{align}
where the last inequality holds since the second summand on the second line amounts to at most $\frac 3{100}$ of the first summand.
Finally, estimates \eqref{nokBin} and \eqref{kSum} together imply that

 \begin{align*}
 \sum_{k=0}^{\frac{m}{2}+1}& |a_{2k}(m)| \Big|\int_0^\infty J_n(r)J_{n+m}(r) 
r^{-2k}
\cos(4r) r^{-1} \, dr\Big| \\
&\leq
\Big(\sum_{k=0}^{\frac{m}{2}+1}  \frac{\Gamma(m+2k+1/2)}{\Gamma(m-2k+1/2)\Gamma(4k+1)}\Big)
\Big(\frac{\Gamma(2n+m)}{\Gamma(n+1)\Gamma(n+m+1)}\Big)4^{-2n-m}\\
&\leq
\Big(\frac{103}{100}  \frac{2 e^{1/24}}{\sqrt{\pi}} (m+1)^{-1/2}2^m \Big( \frac{m}{2}+1 \Big) 2^{m/3} A^{m+1}\Big)
\Big(\frac{e^{1/24}}{2\sqrt{\pi}}\frac{1}{n(n+m)^{1/2}}  2^{2n+m} 
\Big)4^{-2n-m}\\
&\leq
n^{-1}2^{m/3} A^{m} 2^{-2n}.
 \end{align*}
 Since $m\leq n$, we finally get the desired estimate:
 
 $$\leq 
n^{-1}2^{n/3} A^{n} 2^{-2n}
=
n^{-1}\Big(\frac{2^{1/3} A}{4} \Big)^n
\leq
n^{-1} 0.35^n.
$$
This completes the estimate of sum $(i)$ with $\alpha=1$. For the other cases, letting $k\in\{0,1,\ldots,m/2+1\}$ and $1\leq j\leq 7$, one just checks that the bounds given by Lemma \ref{L2}, namely

$$\left|\int_0^\infty J_n(r)J_{n+m}(r) 
r^{-2k-j} e^{4ir}  \, dr\right|
\leq
\frac{2^{2k+j-1}}{4^{2n+m}}\frac{(2n+m-2k-j)!}{n!(n+m)!},
$$ 
are decreasing in $j$ as long as the conditions of the statement are met.
\end{proof}

\noindent {\it Remark.}
We will need the following observation for the purpose of our applications. For $n\geq 20$, we have that ${0.35}^{n/2}\leq n^{-3}$, and so the bound given by Proposition \ref{Careful4rEstimate} can be further estimated as follows:

$$n^{-1} 0.35^n=n^{-1} {0.35}^{n/2} {0.35}^{n/2}\leq  {0.35}^{n_0/2} n^{-4}
\leq 
  0.6^{n_0}  n^{-4},$$
provided $n\geq n_0= 20$.
Alternatively, still for $n\geq 20$, we have that $0.35^{\tau n}\leq n^{-5}$ if $\tau>0.72$. Using this bound instead, we see that

$$n^{-1} 0.35^n=n^{-1} {0.35}^{\tau n} {0.35}^{(1-\tau)n}\leq  {0.35}^{(1-\tau)n_0} n^{-6}\leq  0.75^{n_0}  n^{-6}.$$
All in all, we have the following upper bound for the quantities considered in Proposition  \ref{Careful4rEstimate}:
$$\Big(\min\{  0.6^{n_0} ,  0.75^{n_0}  n^{-2}\} \Big) n^{-4}.$$
This distinction will play a role to ensure good bounds for the $m=4$ terms which were expanded to one higher order in the last subsection. 
\vspace{.2cm}

We are finally ready to estimate the contribution coming from the oscillatory terms $(\beta_0+\beta_2 t^2+\beta_4 t^4)\cos(4r)$ and $(\gamma_1 t+\gamma_3 t^3+\gamma_5 t^5)\sin(4r)$ in expression \eqref{MuCos}, and similarly in \eqref{MuSin}. Appealing to Proposition \ref{Careful4rEstimate} and the remark following it, we see that we can take the following for errors of the second kind:

$$E_{2,*}^{(\cos)}(m,n):=\frac{1}{8}\Big(|\beta_0|+\frac{|\gamma_1|}{16}+\frac{|\beta_2|}{16^2}+\frac{|\gamma_3|}{16^3}+\frac{|\beta_4|}{16^4}+\frac{|\gamma_5|}{16^5}\Big)  \theta^{n_0}  n^{-t};$$
$$E_{2,*}^{(\sin)}(m,n):=\frac{1}{8}\Big(|\gamma_0|+\frac{|\beta_1|}{16}+\frac{|\gamma_2|}{16^2}+\frac{|\beta_3|}{16^3}+\frac{|\gamma_4|}{16^4}+\frac{|\beta_5|}{16^5}\Big)  \theta^{n_0}  n^{-t},$$
where  $(\theta, t)= (0.75, 6)$ if $m=4$ and $(\theta, t)= (0.6, 4)$ if $m\neq 4$. Plugging the values of $\beta,\gamma$ from \eqref{betasgammasCos} and \eqref{betasgammasSin}, we obtain the estimates

\begin{equation}\label{EstErrorSecondKind4}
|E_{2,*}^{(\cos)}(4,n)|, |E_{2,*}^{(\sin)}(4,n)|\leq \begin{cases}
    0.39 \cdot 0.75^{n_0} n^{-6}, & \textrm{if $*=0$},\\
   0.30 \cdot 0.75^{n_0} n^{-6}, & \textrm{if $*=1$},
  \end{cases} 
\end{equation}
and,  if $m\neq 4$, 
\begin{equation}\label{EstErrorSecondKindNot4}
|E_{2,*}^{(\cos)}(m,n)|, |E_{2,*}^{(\sin)}(m,n)|\leq \begin{cases}
    0.39 \cdot 0.6^{n_0} n^{-4}, & \textrm{if $*=0$},\\
    0.30 \cdot 0.6^{n_0} n^{-4}, & \textrm{if $*=1$}.
  \end{cases} 
\end{equation}


\section{Putting it all together}\label{sec:puttingtogether}
In the last section we analyzed the core integrals, which were decomposed into main terms and error terms. We derived several estimates which are recalled below in each case. These are  used together with Estimates A and B to yield appropriate bounds, which are then evaluated at $n_0=20$:

\begin{itemize}
\item[(i)] If $m=0$, then we use the knowledge about the main terms coming from \eqref{MainCos0} and  \eqref{MainSin0}, the estimates for the error terms of the first kind contained in \eqref{E1Cos0} and \eqref{E1Sin0}, and the 
bounds for the error terms of the second kind from \eqref{EstErrorSecondKindNot4}, to conclude that
\begin{multline*}
\left|\int_0^\infty J_n^2(r) \mathfrak{J}_0^4(r)r^{-1}dr-\frac{3}{16}\frac{1}{n}+\frac{3}{128}\frac{1}{(n-1)n(n+1)}\right|\\
\le 
\Big((0.022+0.026+0.0016)n_0^{-1}+0.74n_0^{-5/2}+0.78\cdot 0.6^{n_0}\Big)n^{-4}
\leq
0.0030 n^{-4},
\end{multline*}
and that
\begin{multline*}
\left|\int_0^\infty J_n^2(r) \mathfrak{J}_1^2(r)\mathfrak{J}_0^2(r)r^{-1}dr-\frac{1}{16}\frac{1}{n}-\frac{3}{128}\frac{1}{(n-1)n(n+1)}\right|\\
\le 
\Big((0.023+0.015+0.0030)n_0^{-1}+1.12n_0^{-5/2}+0.60\cdot 0.6^{n_0}\Big)n^{-4}
\leq
0.0028 n^{-4}.
\end{multline*}

\item[(ii)] If $m=2$, then we use the knowledge about the main terms coming from \eqref{MainCos2} and  \eqref{MainSin2}, the estimates for the error terms of the first kind contained in \eqref{E1Cos2} and \eqref{E1Sin2}, and the 
bounds for the error terms of the second kind from \eqref{EstErrorSecondKindNot4}, to conclude that
\begin{multline*}
\left|\int_0^\infty J_{n+2}(r)J_n(r) \mathfrak{J}_2(r)\mathfrak{J}_0^3(r)r^{-1}dr-\frac{15}{256}\frac{1}{n(n+1)(n+2)}\right|\\
\le 
\Big((0.039+0.0062)n_0^{-1}+0.74n_0^{-5/2}+0.162n_0^{-3}+0.78\cdot 0.6^{n_0}\Big)n^{-4}
\leq
0.0028n^{-4},
\end{multline*}
and that
\begin{multline*}
\left|\int_0^\infty J_{n+2}(r)J_n(r) \mathfrak{J}_2(r)\mathfrak{J}_1^2(r)\mathfrak{J}_0(r)r^{-1}dr-\frac{9}{256}\frac{1}{n(n+1)(n+2)}\right|\\
\le 
\Big((0.012+0.0031)n_0^{-1}+1.12n_0^{-5/2}+0.166n_0^{-3}+0.60\cdot 0.6^{n_0}\Big)n^{-4}
\leq
0.0015n^{-4}.
\end{multline*}

\item[(iii)] If $m=4$, then we use the knowledge about the main terms coming from \eqref{MainCos4} and  \eqref{MainSin4}, the estimates for the error terms of the first kind contained in \eqref{E1Cos4} and \eqref{E1Sin4}, and the 
bounds for the error terms of the second kind from \eqref{EstErrorSecondKind4}, to conclude that
\begin{multline*}
\left|\int_0^\infty J_{n+4}(r)J_n(r) \mathfrak{J}_4(r)\mathfrak{J}_0^3(r)r^{-1}dr-\frac{1557}{4096}\frac{1}{n(n+1)(n+2)(n+3)(n+4)}\right|\\
\le 
\Big(0.74n_0^{-1/2}+(0.42+0.086) n_0^{-1}+2.823n_0^{-3}+0.78\cdot 0.75^{n_0}\Big)n^{-6}
\leq
0.197 n^{-6},
\end{multline*}
and that
\begin{multline*}
\left|\int_0^\infty J_{n+4}(r)J_n(r) \mathfrak{J}_4(r)\mathfrak{J}_1^2(r)\mathfrak{J}_0(r)r^{-1}dr-\frac{855}{4096}\frac{1}{n(n+1)(n+2)(n+3)(n+4)}\right|\\
\le 
\Big(1.12n_0^{-1/2}+(0.11+0.063)n_0^{-1}+2.885n_0^{-3}+0.60\cdot 0.75^{n_0}\Big)n^{-4}
\leq
0.264 n^{-6}.
\end{multline*}

\item[(iv)] If $m\geq 6$ is even, then there are no main terms, and we use the estimates for the error terms of the first kind contained in \eqref{E1Cosm} and \eqref{E1Sinm}, and the 
bounds for the error terms of the second kind from \eqref{EstErrorSecondKindNot4}, to conclude that
\begin{multline*}
\left|\int_0^\infty J_{n+m}(r)J_n(r) \mathfrak{J}_m(r)\mathfrak{J}_0^3(r)r^{-1}dr\right|\\
\le 
\Big(0.015 n_0^{-1}+0.74n_0^{-5/2}+(6.34+1.49)n_0^{-3}+0.78\cdot 0.6^{n_0}\Big)n^{-4}
\leq
0.0022 n^{-4},
\end{multline*}
and that
\begin{multline*}
\left|\int_0^\infty J_{n+m}(r)J_n(r) \mathfrak{J}_m(r)\mathfrak{J}_1^2(r)\mathfrak{J}_0(r)r^{-1}dr\right|\\
\le 
\Big(0.015n_0^{-1}+1.12n_0^{-5/2}+(0.09+4.08)n_0^{-3}+0.60\cdot 0.6^{n_0}\Big)n^{-4}
\leq
0.0020 n^{-4}.
\end{multline*}
\end{itemize}
To get the constants promised by Theorem \ref{thm}, one just multiplies the far right-hand sides of each inequality by the normalizing factor $\frac4 {\pi^2}< \frac 1 2$. This concludes the proof of Theorem \ref{thm} for $n\geq 20$.


\section{Numerical estimates for $n<20$}\label{sec:numerical}

In this section, we numerically evaluate the integrals $I_0$ and $I_1$ defined in \eqref{firstintegral}
and \eqref{secondintegral}, respectively,
for $2\le n\le 19$ and even $0\le m\le n$.
We split the integrals into
$$I_j=I_{j,\rm low}+I_{j,\rm high}=\int_0^R \dots dr +\int_R^\infty \dots dr\ .$$
We use a quadrature rule for the first integral and estimate the
second integral by analytic methods. We aim at an absolute error of
at most $0.9 \times 10^{-8}$ for $I_0$ and $I_1$.

The high integral would be entirely 
negligible at our desired accuracy for the threshold (say) $R=10^{10}$, but this would put unnecessarily much computing time
on the low integrals.
We choose $R=63000$ and estimate the high integrals by a more careful analysis of the asymptotic expansion. To bring down the computing time for the low integrals, we use a high degree Newton-Coates quadrature rule. 

We first discuss the high integrals and begin with $I_{0,\rm high}$. 
Since $R$ is large compared to $(n+m)^2$, we take advantage of the 
asymptotic information in Corollary \ref{zlargernsquare}. Splitting each Bessel function into main
term plus error, and applying the distributive law, yields one main integral
of the form
$$I_{0,\rm main,high}=\int_R^\infty \left(\frac{2}{\pi r}\right)^3 
\cos(\omega_{n+m})\cos(\omega_{n})\cos(\omega_{m})\cos^3(\omega_{0})rdr $$
plus $2^6-1$ error terms.

If $n$ is even, since $m$ is even as well, an even number of the integers 
$n,m,n+m$ is congruent two modulo four, and we obtain,
with the periodicity $\cos(\omega_n)=-\cos(\omega_{n+2})$,
$$I_{0,\rm main,high}=I_{0,\rm main,high, even}:=
\int_R^\infty \left(\frac{2}{\pi r}\right)^3 
\cos^6(\omega_{0})rdr\ ,$$
a term which is in fact independent of the particular even $n$ and $m$.
If $n$ is odd, then we obtain similarly
$$I_{0,\rm main,high}=I_{0,\rm main,high, odd}:=
\int_R^\infty \left(\frac{2}{\pi r}\right)^3 
\cos^2(\omega_{1})\cos^4(\omega_{0})rdr\ .$$
Likewise, if $n$ is even, we have 
$$I_{1,\rm main,high}=I_{1,\rm main,high, even}:=
\int_R^\infty \left(\frac{2}{\pi r}\right)^3 
\cos^2(\omega_{1})\cos^4(\omega_{0})rdr\ ,$$
and, if $n$ is odd, 
$$I_{1,\rm main,high}=I_{1,\rm main,high, odd}:=
\int_R^\infty \left(\frac{2}{\pi r}\right)^3 
\cos^4(\omega_{1})\cos^2(\omega_{0})rdr\ .$$

These integrals have closed-form expressions in terms of trigonometric
and trigonometric integral functions. {\it Mathematica} calculates these expression
and evaluates them with prescribed accuracy, resulting in
$$
|I_{0,\rm main, high, even}-1.2798\times 10^{-6}| <10^{-10}, $$
$$|I_{0,\rm main, high, odd}-0.2560\times 10^{-6}|<10^{-10},$$
$$|I_{1,\rm main, high}-0.2560\times 10^{-6}|<10^{-10},$$
where the distinction between even and odd $n$ is not visible at the 
prescribed accuracy in the case of $I_{1,\rm main, high}$.
A sample {\it Mathematica} code used to evaluate $I_{0,\rm main,high, even}$ is the following:
$$\rm N[Integrate[(2/Pi)^3*
   Cos[r - Pi/4]^6 *r^{(-2)}, \{r, 63000, Infinity\}], 
  20]\ .$$

We now estimate the $2^6-1$ error terms of $I_{i,\rm high}-I_{i,\rm main, high}$.
Of these error terms, six of them consist of an integral of a product of five
main terms of Corollary \ref{zlargernsquare} and one error term
of Corollary \ref{zlargernsquare}.
To estimate these six terms, we use the finer information from Corollary 
\ref{zlargernsquarefine} for the error term of Corollary \ref{zlargernsquare}.

The second main term of Corollary \ref{zlargernsquarefine} leads to integrals 
of the type
$$-\frac{4((n+m)^2-1)}8 \Big|\int_R^\infty \left(\frac{2}{\pi r}\right)^3 
\sin(\omega_{n+m})\cos(\omega_{n})\cos(\omega_{m})\cos^3(\omega_{0})dr\Big|$$
and similar terms with a different cosine factor  replaced by a sine factor 
and corresponding prefactor. The product of the six trigonometric functions is odd about the point $\pi/4$.
Thus this product integrates to $0$ over each period. On the period 
$[R+2\pi k, R+2\pi(k+1))$ with any nonnegative integer $k$, we may thus replace the weight 
$r^{-3}$ by the difference between $r^{-3}$ and its mean over that interval. This difference is bounded by $6r^{-4}\pi $ on that interval, hence we may estimate the sum of terms arising
from the second main term of Corollary \ref{zlargernsquarefine} by
$$ 3 \pi ((37)^2+(19)^2+(18)^2+ 3) \int_R^\infty \left(\frac{2}{\pi}\right)^3 r^{-4}dr \le 2.1\times 10^{-11}\ .$$
The sum of the six terms arising from the error terms of Corollary \ref{zlargernsquarefine}
can be further estimated by
$$ \frac 14 ((37)^4+(19)^4+(18)^4+ 3) \int_R^\infty \left(\frac{2}{\pi}\right)^3 r^{-4}dr \le 1.64\times 10^{-9}\ .$$

Next come fifteen terms of the original $2^6-1$ error terms which have four main terms and
two error terms of Corollary \ref{zlargernsquare}. These benefit from an integration of 
the negative fourth power of $r$, and can be estimated by
$$[37^2\times 19^2+37^2\times 18^2 +19^2\times 18^4 + 12\times 36^2]\int_R^\infty \left(\frac{2}{\pi}\right)^3  r^{-4} dr\le 3.32\times10^{-9}\ .$$
The remaining $2^6-1-6-15=42$ terms benefit from an integration of at least the negative fifth power of $r$, and are estimated even more crudely as
$$ 42\times [37^2\times 19^2\times 18^2] \int_R^\infty \left(\frac{2}{\pi}\right)^3 r^{-5} dr\le 4.5\times10^{-10}\ . $$
Adding all these error contributions yields
$$|I_{i,\rm high}-I_{i,\rm main, high}|\le 5.5 \times10^{-9}\ .$$

We next turn to the low integrals. We recall the Newton-Coates rule
$$\int_0^6 f(x)\, dx=F(f)$$
with
$$F(f)=140^{-1}(41f(0)+216f(1)+27f(2 )+272f(3)+27f(4)+216f(5)+41f(6)),$$
which is valid for all real polynomials $f$ up to degree $7$.
For any eight times continuously differentiable function $f$ on $[0,6]$, we have that 

\begin{equation}\label{NCError}
\Big|\int_0^6 f(x)\, dx-F(f)\Big|\le  \frac{6^4}{5} \sup_{\xi\in [0,1]}\frac{|f^{(8)}(\xi)|}{8!}\ .
\end{equation}
A well-known argument shows that 
polynomials of degree eight extremize this inequality. It is then a straightforward matter of checking that polynomials of degree eight, whose eighth derivative is constant, realize the optimal constant $\frac{6^4}{5}$ promised by \eqref{NCError}.

Now let $F_{a,w}$ be the suitably scaled and translated Newton-Coates formula which integrates polynomials of degree $7$ on the interval $[a,a+6w]$ exactly. Then, by rescaling, 

$$\Big|\int_a^{a+6w} f(x)\, dx-F_{a,w}(f)\Big|\le  w^9\frac{6^4}{5}
\sup_{\xi\in [a,a+6w]}\frac {|f^{(8)}(\xi)|}{8!}\ .$$
Now assume that the length of the interval $[a,b]$ is an integer multiple of $6w$,
say $6wN$. Then partitioning this interval into $N$ intervals of length $6w$ and 
applying the Newton-Coates formula on each interval yields
 
$$\Big|\int_a^{b} f(x)\, dx-\sum_{k=0}^{N-1}F_{a+kw,w}(f)\Big|\le  (b-a) w^8\frac{6^3 }{5}
\sup_{\xi\in [a,b]} \frac{|f^{(8)}(\xi)|}{8!}\ .$$

We cut the interval $[0,R]$ into $[0,S]\cup [S,R]$ with
$S=3600$. On the interval $[0,S]$, we estimate the eighth derivative of the functions
$$f(r)=J_{n+m}(r) J_n(r) J_m (r) J_0^3(r)r$$
and
$$f(r)=J_{n+m}(r) J_n(r) J_m (r) J_1^2(r)J_0(r)r$$
using the Cauchy integral formula for the circle of radius $1$ about $r$ 
together with the trivial bound \eqref{EstJnIm} to obtain the estimate
$$|f^{(8)}(r)|\le 8! e^6 (S+1)\ .$$
Approximating the integral over $[0,S]$ by the above summation rule with width 
$w=0.003$ gives the error bound
$$S (.003)^8 \frac{6^3 }{5} 
e^6(S+1) \le 1.49\times 10^{-9}\ .$$

On the interval $[S,R]$, we estimate the eighth derivative of $f$ again by the Cauchy
integral formula with circles of radius one. We use the estimate from Corollary 
\ref{zlargernsquare}
for $J_n^+$ to obtain, for $\Re(z)>S-1$ and $\Im(z)\le 1$,
$$|J_n^+(z)|\le \Big(\frac{2}{\pi|z|}\Big)^{1/2} (1+n^2/S)\cosh(\Im(z))\times 1.01 
\ ,$$
and similarly for $J_n^-$. Estimating the product
of the various terms analogous to $(1+n^2/S)\times 1.01$ by $3$, we then obtain
$$|f^{(8)}(r)|\le 3\times 8!\times  \Big(\frac{2}{\pi(S-1)}\Big)^3 (\cosh(1))^6(R+1)\ .$$

Approximating the integral over $[S,R]$ by the above summation rule with width 
$w=0.05$ gives the error bound
$$
3\times (R-S) w^8 \frac{6^3 }{5}  \Big(\frac{2}{\pi(S-1)}\Big)^3 (\cosh(1))^6(R+1)\le 
1.42\times 10^{-9}\ .$$

Collecting error terms, we obtain 

$$|I_0- 1.2798\times 10^{-6}-F_{[0,S]}- F_{[S,R]}|\le 0.85\times 10^{-8}$$
if $n$ is even, and 
$$|I_0- 0.256\times 10^{-6}-F_{[0,S]}- F_{[S,R]}|\le 0.85\times 10^{-8}$$
if $n$ is odd, and 
$$|I_1- 0.256\times 10^{-6}-F_{[0,S]}- F_{[S,R]}|\le 0.85\times 10^{-8}$$
for any $n$, where $F_{[0,S]}$ and $F_{[S,R]}$ are the 
quadrature formulae described above for the corresponding integrals.

We evaluate $F_{[0,S]}$ and $F_{[S,R]}$ using {\it Mathematica}. Products of Bessel functions at the grid points are computed with 20-digit precision, and the corresponding rounding errors for $F_{[0,S]}+F_{[S,R]}$ can be safely estimated by $0.05\times 10^{-8}$.
As an example,
in the case $n=14$ and $m=4$ for $I_0$, we use the following code to compute $F_{[0,S]}$:

\begin{eqnarray*} 
&\rm BJ[x{\_}] :=
 N[BesselJ[18, x]*BesselJ[14, x]*BesselJ[4, x]*BesselJ[0, x]^3*
   x, 20]\\
&\rm BJSA := 41*BJ[0] +\\
 & \rm82*Sum[BJ[x], \{x, 18/1000, 3599982/1000, 18/1000\}] +\\
 &  \rm216*Sum[BJ[x], \{x, 3/1000, 3599985/1000, 18/1000\}] +\\
 & \rm27*Sum[BJ[x], \{x, 6/1000, 3599988/1000, 18/1000\}] +\\
 & \rm272*Sum[BJ[x], \{x, 9/1000, 3599991/1000, 18/1000\}] +\\
 & \rm27*Sum[BJ[x], \{x, 12/1000, 3599994/1000, 18/1000\}] +\\
 & \rm216*Sum[BJ[x], \{x, 15/1000, 3599997/1000, 18/1000\}] +\\
 & \rm41*BJ[3600]\\
&\rm(BJSA*.018)/840
\end{eqnarray*}

 For $m=0$ and even $n$, Table \ref{table:bigtable} lists  
upper bounds for the quantities
$$\Big(\Big|\frac{3}{4\pi^2}\frac{1}{n}-\frac{3}{32 \pi^2}\frac{1}{(n-1)n(n+1)}
- 1.2798\times 10^{-6}-F_{[0,S]}- F_{[S,R]}\Big|+ 0.9\times10^{-8}\Big)100n^4$$
on top of each entry, and for
$$\Big(\Big|\frac{1}{4\pi^2}\frac{1}{n}+\frac{3}{32 \pi^2}\frac{1}{(n-1)n(n+1)}
- 0.256\times 10^{-6}-F_{[0,S]}- F_{[S,R]}\Big|+ 0.9\times 10^{-8}\Big)100n^4$$
at the bottom of each entry, with the appropriate quadrature formulae $F_{[0,S]}$
and $F_{[S,R]}$ described above. For $m=0$ and odd $n$, it similarly lists upper bounds for
$$\Big(\Big|\frac{3}{4\pi^2}\frac{1}{n}-\frac{3}{32 \pi^2}\frac{1}{(n-1)n(n+1)}
- 0.256\times 10^{-6}-F_{[0,S]}- F_{[S,R]}\Big|+ 0.9\times 10^{-8}\Big)100n^4$$
on top, and for
$$\Big(\Big|\frac{1}{4\pi^2}\frac{1}{n}+\frac{3}{32 \pi^2}\frac{1}{(n-1)n(n+1)}
- 0.256\times 10^{-6}-F_{[0,S]}- F_{[S,R]}\Big|+ 0.9\times 10^{-8}\Big)100n^4$$
at the bottom. Thus each entry on the first column ($m=0$) of Table \ref{table:bigtable}, divided by $100$,
provides a constant $c$ for which the estimate of Theorem \ref{thm} 
holds for the corresponding $n$ with $c$ in place of $0.002$ or $0.0015$.
The entries of Table \ref{table:bigtable} for $m>0$ are analogous.

The poorer constants near $n=19$ are artificial and due to the chosen numerical accuracy 
$0.9\times 10^{-8}$; note that, for this value of $n$, the quantity $0.9\times 10^{-8}n^4$ is already close to $0.0014$.
The very good constants at $m=4$ are due to the extra term in the expansion
that has been elaborated in that case.

\begin{center}
\begin{table}
\begin{tabular}{l||l|l|l|l|l|l|l|l|l|l|}

$n$ $\backslash$ $m$& 0 & 2 & 4  & 6 & 8 & 10 & 12 & 14 & 16 & 18\\  \hline 
2 & .85&.14 & & & & & & & &\\
 &.64 &.03 & & & & & & & &\\\hline 
3 &.44 &.16 & & & & & & & &\\
 &.21 &.05 & & & & & & & &\\\hline 
4 &.33 &.16 &.03 & & & & & & &\\
 &.16 &.04 &.01 & & & & & & &\\\hline 
5 &.26 &.15 &.02 & & & & & & &\\
 &.12 &.04 &.01& & & & & & &\\\hline 
6 &.22 &.15 &.02 &.11 & & & & & &\\
 &.10 &.04 &.01 &.06 & & & & & &\\\hline 
7 &.19 &.14 &.02 &.09 & & & & & &\\
 &.09 &.04 &.01 &.05 & & & & & &\\\hline 
8 & .17 &.13 &.02 &.08 &.02 & & & & &\\
 &.08 &.04 &.01 &.05 &.02 & & & & &\\\hline 
 9 &.15 &.13 &.02 &.07 &.02 & & & & &\\
 &.07 &.04 &.01 &.05 &.02 & & & & &\\\hline 
10 & .14 &.13 &.02 &.07 &.02 &.02 & & & &\\
 &.07 &.04 &.02 &.04 &.02 &.02 & & & &\\\hline 
 11& .13 &.13 &.02 &.07 &.03 &.02 & & & &\\
 &.07 &.04 &.02 &.04 &.02 &.02 & & & &\\\hline 
 12& .13 &.13 &.03 &.07 &.03 &.03 &.03 & & &\\
 &.07 &.05 &.03 &.05 &.03 &.03 &.03 & & &\\\hline 
 13& .13 &.13 &.04 &.07 &.04 &.03 &.03 & & &\\
 &.08 &.05 &.03 &.05 &.04 &.03 &.03 & & &\\\hline 
 14& .13 &.13 &.05 &.07 &.05 &.04 &.04 &.04 & &\\
 &.08 &.06 &.04 &.06 &.05 &.04 &.04 &.04 & &\\\hline 
 15& .14 &.14 &.06 &.08 &.06 &.06 &.06 &.06 & &\\
 &.09 &.07 &.06 &.07 &.06 &.06 &.06 &.06 & &\\\hline 
 16& .15&.15 &.07 &.09 &.07 &.07 &.07 &.07 &.07 &\\
 &.10 &.09 &.07 &.08 &.07 &.07 &.07 &.07 &.07 &\\\hline 
 17& .16 &.17 &.09 &.11 &.09 &.09 &.09 &.09 &.09 &\\
 &.12 &.10 &.09 &.10 &.09 &.09 &.09 &.09 &.09 &\\\hline 
 18& .18 &.19 &.11 &.13 &.11 &.11 &.11 &.11 &.11 &.11\\
 &.14 &.13 &.11 &.12 &.11 &.11 &.11 &.11 &.11 &.11\\\hline 
 19& .20 &.20 &.14 &.15 &.14 &.14 &.14 &.14 &.14 &.14\\
 &.16 &.15 &.14 &.14 &.14 &.14 &.14 &.14 &.14 &.14\\\hline 
\end{tabular}
\caption{}
\label{table:bigtable}
\end{table}
\end{center}

%


\end{document}